\documentclass[12pt]{article}
\usepackage{amsfonts}
\usepackage{amsmath, graphicx, amsfonts,amssymb, calrsfs}
\usepackage{amsfonts,mathrsfs, color, amsthm,}
\usepackage{bm}
\usepackage[numbers,sort&compress]{natbib}
\usepackage{xcolor}
\usepackage{amsmath, graphicx, amsfonts,amssymb, calrsfs}
\usepackage{amsfonts,mathrsfs, color, amsthm,mathrsfs,float}
\providecommand{\hK}[1]{\|#1\|_{K^*}}

\renewcommand{\hK}[1]{\|#1\|_{K}}

\usepackage{hyperref}
\usepackage[mathlines]{lineno}
\allowdisplaybreaks

\usepackage{enumitem}
\usepackage{hyperref}
\hypersetup{
  colorlinks   = true, 
  urlcolor     = blue, 
  linkcolor    = blue, 
  citecolor   = red 
  }

 \def\R{\mathbb{R}}

\def\Ksc{\mathcal{K}^{ss}_{(o)}}

\newtheorem{theorem}{Theorem}[section]
\newtheorem{lemma}{Lemma}[section]
\newtheorem{remark}{Remark}[section]
\newtheorem{proposition}{Proposition}[section]
\newtheorem{corollary}{Corollary}[section]

\newtheorem{definition}{Definition}[section]

\def\bt{\begin{theorem}}
\def\et{\end{theorem}}
\def\bl{\begin{lemma}}
\def\el{\end{lemma}}
\def\br{\begin{remark}}
\def\er{\end{remark}}
\def\bc{\begin{corollary}}
\def\ec{\end{corollary}}
\def\bd{\begin{definition}}
\def\ed{\end{definition}}
\def\bp{\begin{proposition}}
\def\ep{\end{proposition}}

\def\RN{\mathbb{R}^N}
\makeatletter
\renewcommand{\@makefntext}[1]{%
  \noindent
  \makebox[0pt][r]{\@makefnmark\ }#1%
}
\makeatother

\title{
Anisotropic Caffarelli-Kohn-Nirenberg inequalities for the
Minkowski functional: $L^p$-remainder identities and extremals
}

\author{Zhenzhen Wei}

\date{}

\begin{document}

\maketitle

\begingroup
\renewcommand{\thefootnote}{}
\footnotetext{%
}
\endgroup

\begin{abstract}
We establish exact anisotropic $L^p$-Hardy and
Caffarelli-Kohn-Nirenberg identities associated with the Minkowski
functional $\|\cdot\|_K$ and the anisotropic radial derivative
$\mathcal R_K$, where $1<p<\infty$ and $K\subset\RN$ is a smooth
and strictly convex body containing the origin in its interior,
not necessarily origin-symmetric. These identities contain explicit
nonnegative $R_p$-remainders and yield sharp anisotropic radial
$L^p$-Hardy and $L^p$-Caffarelli-Kohn-Nirenberg inequalities.
In the noncritical parameter regions, we characterize the full
family of extremal functions in the natural weighted completion
space, while in the critical case the sharp constant is not attained
by any nonzero function in that space. As applications, we establish
sharp anisotropic $L^p$-Heisenberg uncertainty principles and obtain
explicit anisotropic Gaussian-type optimizers. Finally, letting $K^*$ denote the polar body of $K$, we derive
norm-based anisotropic gradient inequalities associated with the
origin-symmetric body $K^*\cap(-K^*)$. When $K$ is
origin-symmetric, the corresponding gradient constants are sharp,
including those in the critical case.
\end{abstract}

\section{Introduction}
The Caffarelli-Kohn-Nirenberg (CKN) inequalities were introduced by Caffarelli, Kohn and Nirenberg
\cite{CaffarelliKohnNirenberg}. They form a fundamental class of weighted
interpolation inequalities in $\mathbb R^N$, and contain, as special or
limiting cases, several classical inequalities, such as the Hardy inequality,
the Sobolev inequality, the Gagliardo-Nirenberg inequality, Hardy-Sobolev-type
inequalities, and the weighted uncertainty principles. These inequalities
play important roles in elliptic partial differential equations, mathematical
physics, spectral theory, geometric analysis, and the calculus of variations.

A central problem in the CKN theory is to determine the sharp constants,
decide whether the sharp constants are attained, and describe the corresponding
extremal functions. In the Euclidean setting, the existence, nonexistence, and
symmetry properties of the extremal functions depend on the parameters; see,
e.g.,
\cite{CatrinaWang,CaldiroliMusina,FelliSchneider,
DolbeaultEstebanLossTarantello,DolbeaultEstebanLoss,WangWillem,ZhongZou}.
Further developments on sharp weighted forms, explicit optimizers, and related
extremal problems can be found in
\cite{CatrinaCosta,LamLuSharpCKN,LamGeneralSharpWeightedCKN,DongLamLu}.
The sharp constants and optimizers for several related classical inequalities
have also been widely studied. For example, the sharp Sobolev constant was
obtained by Talenti \cite{Talenti}, the sharp Hardy-Littlewood-Sobolev
constant was obtained by Lieb \cite{Lieb}, and sharp Gagliardo-Nirenberg
inequalities and their mass-transport proofs were studied in
\cite{Agueh,DelPinoDolbeault,CorderoNazaretVillani,Nguyen}.

Hardy-type inequalities provide an important source of motivation for CKN
inequalities. A general $L^p$-Hardy-type inequality has the form
\begin{align*}
\int_{\Omega}
A(x)|\nabla u|^pdx
\ge
\int_{\Omega}
B(x)|u|^pdx,
\end{align*}
where $\Omega$ is an open domain in $\mathbb R^N$, and $A$ and $B$
are suitable weight functions. From the point of view of identity methods,
Hardy inequalities may be regarded as non-optimal or non-scale-invariant
forms of the corresponding CKN inequalities, while CKN inequalities arise by
optimizing suitable parameters in Hardy-type identities. Improved Hardy
inequalities, ground-state representations, Bessel-pair methods, and related
identity approaches have been studied in
\cite{BrezisVazquez,FrankSeiringer,GhoussoubMoradifam,DuyLamLu}.

A recent and effective approach to Hardy and CKN inequalities is based on exact
identities with nonnegative remainders. Such identities provide a direct way
to determine sharp constants and equality equations, and they also explain the
existence or nonexistence of optimizers through the vanishing of the
corresponding remainders. In the $L^2$ setting, Cazacu, Flynn, Lam and Lu
\cite{CazacuFlynnLamLu} developed this point of view and used Hardy and CKN
identities to derive sharp inequalities, characterize the equality cases, and
establish stability results. The method was later extended to the $L^p$
setting by Do, Flynn, Lam and Lu \cite{DoFlynnLamLu}. 
A further motivation for the identity method is quantitative stability.
For the Sobolev and isoperimetric inequalities, stability estimates
were obtained in
\cite{BianchiEgnell,CianchiFuscoMaggiPratelli,
FigalliMaggiPratelli,FigalliNeumayer},
and for the Heisenberg uncertainty principle in
\cite{McCurdyVenkatraman,Fathi}.
Since the deficits in the identities below are given by explicit
nonnegative $R_p$-remainders, they are naturally suited to
quantitative stability estimates. We intend to pursue this question
elsewhere.
Related $L^p$-Hardy identity approaches based on Bessel pairs,
monomial weights, and general distance functions can also be found
in
\cite{DLL22,Lam18,LLZ20,FLL25}. The present paper follows this
identity-and-remainder approach, but replaces the Euclidean radial derivative
by the anisotropic radial derivative associated with the Minkowski functional.

In many geometric and analytic problems, the Euclidean norm is not the only
natural choice of distance. Let $K\subset\RN$ be a smooth
and strictly convex body containing
the origin in its interior, and let $\hK{x}$ be the Minkowski functional of
$K$. For $r>0$,
$
\{x\in\RN:\hK{x}\le r\}=rK,
$
and hence $\hK{x}$ provides a natural anisotropic distance to the origin.
The corresponding anisotropic radial derivative is defined by
$$
\mathcal R_K(u)(x)
=
\nabla u(x)\cdot \frac{x}{\hK{x}},
\quad x\ne o.
$$
When $K$ is the Euclidean unit ball, this becomes the Euclidean radial
derivative. For a smooth
and strictly convex body containing the origin in its interior, especially 
when $K$ is not
origin-symmetric, the Minkowski functional need not be even.

Anisotropic inequalities involving origin-symmetric convex bodies and their Minkowski functionals have been studied from several viewpoints. In the norm-based Finsler
framework, one usually works with an even, convex, positively one-homogeneous
function and its polar function. This setting has led to anisotropic Sobolev
inequalities, anisotropic isoperimetric inequalities, Hardy inequalities, and
related convex symmetrization methods; see
\cite{AFTL97,Ferone-Volpicelli,VanSchaftingen,
BianchiCianchiGronchi,DellaPietraDiBlasioGavitone,
MercaldoSanoTakahashi}.
Anisotropic CKN-type inequalities and weighted Hardy-Sobolev-type
inequalities have also been studied in several directions, including
the Finsler norm setting and coordinate-splitting weights; see
\cite{LiYan,Shen,BaoChen}. Related CKN-type inequalities involving
quasi-norms and radial derivatives on homogeneous groups were obtained
by Ozawa, Ruzhansky and Suragan \cite{OzawaRuzhanskySuragan}. For
related Hardy identities and inequalities on homogeneous and Carnot
groups, see also \cite{Lam19,FL21}. For further developments related
to anisotropic Sobolev capacities, anisotropic Riesz potentials, and
isoperimetric, capacitary, and Minkowski-type variational inequalities
for convex bodies, we refer to
\cite{XY17,HXY18,Xiao07,Xiao17a,Xiao17b,CNSXYZ15,JWX26,JX26}.

The existing norm-based anisotropic inequalities are often formulated
under an origin-symmetry or norm assumption. In contrast, the
Minkowski functional of a smooth
and strictly convex body containing the origin
in its interior need not be even. This leads to a genuinely
non-symmetric anisotropic setting. 
Let $K^*$ denote the polar body of $K$. In the non-symmetric
setting, the usual estimate
$$
|x\cdot y|
\le
\|x\|_K\|y\|_{K^*}
$$
is no longer available in this form. Instead, one has the one-sided
anisotropic Cauchy-Schwarz inequalities
$$
-\|x\|_K\|-y\|_{K^*}
\le
x\cdot y
\le
\|x\|_K\|y\|_{K^*}.
$$
Let
$$
L=\operatorname{conv}\big(K\cup(-K)\big),
$$ where $\operatorname{conv}$ denotes the convex hull. Then
$
L^*=K^*\cap(-K^*)
$
and
$
\|y\|_{L^*}
=
\max\big\{
\|y\|_{K^*},
\|-y\|_{K^*}
\big\}.
$
Consequently,
$$
|x\cdot y|
\le
\|x\|_K\|y\|_{L^*}.
$$
Thus, the anisotropic radial derivative $\mathcal R_K$ remains the
natural object for the sharp radial CKN inequalities, while the norm
$\|\cdot\|_{L^*}$ provides a natural formulation for the corresponding
full-gradient inequalities in the non-symmetric setting.

The anisotropic $L^2$-CKN inequalities associated with the Minkowski
functional were studied in \cite{Wei}. In that work, the sharp constants,
attainability in the natural completion space, anisotropic Heisenberg-type
inequalities, and max-type anisotropic gradient inequalities were obtained in
the $L^2$ setting. The purpose of the present paper is to develop the
corresponding $L^p$ theory for $1<p<\infty$. The main new difficulty is
that the square completion used in the $L^2$ case has to be replaced by the
nonlinear remainder
$$
R_p(s,t)
=
|t|^p+(p-1)|s|^p-p|s|^{p-2}st,
\quad s,t\in\mathbb R.
$$
By the convexity of the function $t\mapsto |t|^p$, one has
$R_p(s,t)\ge0$, and equality holds if and only if $s=t$. This remainder
is the natural $L^p$ substitute for the square remainder in the $L^2$
theory. By choosing suitable radial weights in the identity and optimizing
the parameter, we obtain exact anisotropic $L^p$-Hardy and $L^p$-CKN
identities, sharp constants, and equality equations in the Minkowski
functional setting.

More precisely, let $K\subset\RN$ be a smooth
and strictly convex body containing the origin in its interior. In the parameter regions
$
b+1-a>0, b\le\frac{N-p}{p},
$
or
$
b+1-a<0, b\ge\frac{N-p}{p}.
$
Indeed,
\begin{align*}
N-1-(p-1)a-b
&=
N-p(b+1)+(p-1)(b+1-a),
\\
(p-1)a+b+1-N
&=
p(b+1)-N+(p-1)(a-b-1).
\end{align*}
Thus, the two parameter regions above guarantee the positivity of
the corresponding constants, although they are not necessary for
positivity.
We prove that
 for
$u\in C_c^\infty(\mathbb R^N\setminus\{o\})$,
\begin{align}\label{f70}
\!\!\bigg(
\int_{\mathbb R^N}
\frac{|\mathcal R_K(u)|^p}{\|x\|_K^{bp}}dx
\bigg)^{\frac1p}
\bigg(
\int_{\mathbb R^N}
\frac{|u|^p}{\|x\|_K^{ap}}dx
\bigg)^{\frac{p-1}{p}}
\ge
\frac{|N-1-(p-1)a-b|}{p}
\int_{\mathbb R^N}
\frac{|u|^p}{\|x\|_K^{(p-1)a+b+1}}dx.
\end{align}
We also note that \eqref{f70} is scale invariant.
Indeed, for
$
u_\tau(x)=u(\tau x), \tau>0,
$
both sides of \eqref{f70} are multiplied by
$
\tau^{(p-1)a+b+1-N}.
$
This also confirms the exponent
$(p-1)a+b+1$ in the weighted integral.
The constant $\frac{|N-1-(p-1)a-b|}{p}$ is sharp. When $a\ne b+1$, the equality equation
follows from the vanishing of the $R_p$-remainder and reduces to a
first-order equation in the anisotropic radial variable
$r=\|x\|_K$. Solving this equation gives the extremal functions.
Since these functions are generally not compactly supported, equality
is not attained in $C_c^\infty(\RN\setminus\{o\})$. We therefore work
in the completion space $\mathcal C_{K,a,b}^p$ associated with the
natural weighted radial norm
$$
\|u\|_{\mathcal C_{K,a,b}^p}
=
\bigg(
\int_{\RN}
\frac{|\mathcal R_K(u)|^p}{\|x\|_K^{bp}}dx
+
\int_{\RN}
\frac{|u|^p}{\|x\|_K^{ap}}dx
\bigg)^{\frac1p},
$$
and characterize the extremal functions in this completion space.
In the critical case $a=b+1$, which belongs to neither of the two
parameter regions above, inequality \eqref{f70} still holds with
the constant
$$
\frac{|N-p(b+1)|}{p},
$$
and is equivalent to the weighted anisotropic $L^p$-Hardy
inequality
$$
\int_{\RN}
\frac{|\mathcal R_K(u)|^p}{\|x\|_K^{bp}}dx
\ge
\bigg|\frac{N-p(b+1)}{p}\bigg|^p
\int_{\RN}
\frac{|u|^p}{\|x\|_K^{p(b+1)}}dx.
$$
The constant
$$
\bigg|\frac{N-p(b+1)}{p}\bigg|^p
$$
is sharp, but it is not attained by any nonzero function in
$\mathcal C_{K,b+1,b}^p$.

As a further consequence of the same identity method, we establish the sharp
anisotropic $L^p$-Heisenberg uncertainty principle associated with
$\|\cdot\|_K$. Combining the radial identity with convex symmetrization
tools, for nonnegative
$$
u\in\mathcal C
:=
W^{1,p}(\RN)\cap
\bigg\{
u:\int_{\RN}\hK{x}^{p'}|u|^pdx<\infty
\bigg\},
\quad p'=\frac{p}{p-1},
$$
we prove
$$
\int_{\mathbb R^N}
\|-\nabla u\|_{K^*}^pdx
+
(p-1)
\int_{\mathbb R^N}
\|x\|_K^{p'}u^pdx
\ge
N
\int_{\mathbb R^N}u^pdx,
$$
and
$$
\bigg(
\int_{\mathbb R^N}
\|-\nabla u\|_{K^*}^pdx
\bigg)^{\frac1p}
\bigg(
\int_{\mathbb R^N}
\|x\|_K^{p'}u^pdx
\bigg)^{\frac1{p'}}
\ge
\frac{N}{p}
\int_{\mathbb R^N}u^pdx.
$$
Both constants are sharp, and the optimizers are anisotropic Gaussian-type
functions generated by the Minkowski functional.

Finally, using the two-sided anisotropic Cauchy-Schwarz inequality
associated with
$$
L=\operatorname{conv}\big(K\cup(-K)\big),
\quad
L^*=K^*\cap(-K^*),
$$
we derive norm-based anisotropic gradient versions of the radial
inequalities, with the gradient term expressed through
$\|\nabla u\|_{L^*}$. In the non-symmetric case, we do not claim
sharpness of these gradient inequalities, since equality in the
gradient control need not be compatible with the $K$-radial
extremal functions. When $K$ is origin-symmetric, one has $L=K$,
and we prove that all the resulting anisotropic gradient constants
are sharp, including those in the critical case $a=b+1$.

The paper is organized as follows. In Section \ref{s2}, we collect basic
facts about convex bodies, the Minkowski functional, anisotropic
radial derivatives, the anisotropic polar formula, and convex
symmetrization. In Section \ref{s3}, we establish anisotropic $L^p$-Hardy
and $L^p$-CKN identities with nonnegative $R_p$-remainders. In
Section \ref{s4}, we first prove the sharp anisotropic $L^p$-Heisenberg
uncertainty principle and exhibit explicit optimizers. We then
establish the sharp anisotropic radial $L^p$-CKN inequalities,
characterize the extremal functions in the natural completion space
in the noncritical cases, and prove nonattainment in the critical
case. We also obtain the radial anisotropic $L^p$-Heisenberg
uncertainty principle as a consequence and derive norm-based
anisotropic gradient inequalities involving $\|\cdot\|_{L^*}$.
When $K$ is origin-symmetric, we further prove the sharpness of the
corresponding gradient inequalities, including the critical case.

\section{Preliminaries}\label{s2}

Throughout this paper, let $N\ge2$ and let $\RN$ be the $N$-dimensional
Euclidean space. We write $o$ for the origin of $\RN$, $x\cdot y$ for the
Euclidean inner product of $x,y\in\RN$, and
$
|x|=(x\cdot x)^{\frac12}
$
for the Euclidean norm. 
If $\Omega\subset\RN$ is an open set, then $C^1(\Omega)$ denotes the
space of real-valued functions on $\Omega$ whose first-order partial
derivatives exist and are continuous on $\Omega$. More generally,
$C^\infty(\Omega)$ denotes the space of real-valued smooth functions on
$\Omega$, and $C_c^\infty(\Omega)$ denotes the space of functions in
$C^\infty(\Omega)$ with compact support in $\Omega$. 

For $1\le p<\infty$, $L^p(\RN)$ denotes the usual Lebesgue space with norm
$$
\|u\|_p
=
\bigg(
\int_{\RN}|u(x)|^pdx
\bigg)^{\frac1p}.
$$
For $1<p<\infty$, we  denote by $W^{1,p}(\RN)$ the  first-order Sobolev space. Thus
$u\in W^{1,p}(\RN)$ if $u\in L^p(\RN)$ and its weak gradient
$
\nabla u=(\partial_1u,\dots,\partial_Nu)
$
belongs to $L^p(\RN;\RN)$.

A convex body is a compact convex subset $K\subset\RN$ with nonempty
interior. Let $\mathcal K$ be the class of all convex bodies in $\RN$,
let $\mathcal K_{(o)}$ be the class of convex bodies containing $o$ in
their interiors, and let $\mathcal K_{(o)}^c$ be the subclass of
$\mathcal K_{(o)}$ consisting of origin-symmetric convex bodies. For
$K\in\mathcal K$, the support function $h_K:\RN\to\mathbb R$ is defined by
$$
h_K(x)=\max\{x\cdot y:y\in K\}.
$$
We write $\partial K$ for the boundary of $K$. An outer unit normal vector
of $K$ at $y\in\partial K$ is a unit vector $\xi$ such that
$$
h_K(\xi)=\xi\cdot y.
$$
The convex body $K$ is called smooth if each boundary point has a unique
outer unit normal vector. It is called strictly convex if $\partial K$ does
not contain a nontrivial line segment.

For $K\in\mathcal K_{(o)}$, the polar body $K^*$ is defined by
$$
K^*
=
\{x\in\RN:x\cdot y\le1\ \text{for all }y\in K\}.
$$
Then $K^*\in\mathcal K_{(o)}$ and $(K^*)^*=K$. The Minkowski functional
of $K$ is defined by
$$
\hK{x}
=
\inf\{\lambda>0:x\in\lambda K\},
\quad x\in\RN.
$$
It is convex, continuous, positively one-homogeneous, and satisfies
$$
K=\{x\in\RN:\hK{x}\le1\}.
$$
Moreover,
\begin{align*}
\hK{x}=h_{K^*}(x),
\quad
\|x\|_{K^*}=h_K(x).
\end{align*}
If $K$ is not origin-symmetric, then $\|\cdot\|_K$ need not be even and
therefore need not be a norm. If $K\in\mathcal K_{(o)}^c$, then
$\|\cdot\|_K$ is a norm.

A function $u:\RN\to\mathbb{R}$ is said to be \textit{radially symmetric} with respect to $K\in\mathcal{K}_{(o)}$, if there exists a function, still denoted by $u$, such that $u(x)=u(\hK{x})$ for any $x\in \RN.$
For $K\in\mathcal K_{(o)}$ and $x\ne o$, let
$$
r=\hK{x},
\quad
\sigma_K(x)=\frac{x}{\hK{x}},
$$
then $\sigma_K(x)\in\partial K$, $\|\sigma_K(x)\|_K=1$, and
$
x=r\sigma_K(x).
$
The map
$$
\mathcal L_K:(0,\infty)\times\partial K\to\RN\setminus\{o\},
\quad
\mathcal L_K(r,\sigma)=r\sigma,
$$
defines the anisotropic polar coordinates associated with $K$.

The anisotropic polar formula can be expressed in terms of the cone
measure of $K$.

\begin{proposition}[Anisotropic polar formula]
\label{prop:anisotropic-polar-formula}
Let $K\in\mathcal K_{(o)}$. The cone measure $\mu_K$ on
$\partial K$ is defined by
$
\mu_K(\omega)
=
N
\big|
\{r\sigma:\sigma\in\omega,\ 0\le r\le1\}
\big|
$
for every Borel set $\omega\subset\partial K$. Then, for every
nonnegative measurable function $f$ on $\RN$,
\begin{align}
\int_{\RN}f(x)dx
=
\int_0^\infty
\int_{\partial K}
f(r\sigma)r^{N-1}d\mu_K(\sigma)dr.
\label{anisotropic-polar-formula}
\end{align}
If $K$ is smooth, then
$$
d\mu_K(\sigma)
=
\big(
\sigma\cdot\nu_K(\sigma)
\big)
d\mathcal H^{N-1}(\sigma),
$$
where $\nu_K(\sigma)$ denotes the outer unit normal to $K$ at
$\sigma\in\partial K$. Moreover,
$$
\mu_K(\partial K)
=
N|K|
\in(0,\infty).
$$
\end{proposition}

\begin{proof}
Assume first that $K$ is smooth. Consider the anisotropic polar
coordinate map
$$
L_K:(0,\infty)\times\partial K
\longrightarrow
\RN\setminus\{o\},
\quad
L_K(r,\sigma)=r\sigma.
$$
Its Jacobian is
$
r^{N-1}
\big(
\sigma\cdot\nu_K(\sigma)
\big).
$
Hence, the area formula gives
\begin{align*}
\int_{\RN}f(x)dx
&=
\int_0^\infty
\int_{\partial K}
f(r\sigma)
r^{N-1}
\big(
\sigma\cdot\nu_K(\sigma)
\big)
d\mathcal H^{N-1}(\sigma)dr,
\end{align*}
which proves \eqref{anisotropic-polar-formula} and the density
formula for $\mu_K$.

Since $o\in\operatorname{int}K$, one has
$
\sigma\cdot\nu_K(\sigma)
=
h_K\big(\nu_K(\sigma)\big)
>0
$
for every $\sigma\in\partial K$. In addition, the divergence theorem
gives
\begin{align*}
\mu_K(\partial K)
=
\int_{\partial K}
\sigma\cdot\nu_K(\sigma)
d\mathcal H^{N-1}(\sigma)=
\int_K\operatorname{div}x dx
=
N|K|.
\end{align*}
The general case follows from the definition of the cone measure and
the corresponding measure-theoretic polar decomposition.
\end{proof}

We denote by
$
L^p(\partial K,d\mu_K)
$
the corresponding $L^p$-space, equipped with the norm
$$
\|\Phi\|_{L^p(\partial K,d\mu_K)}
=
\bigg(
\int_{\partial K}
|\Phi(\sigma)|^pd\mu_K(\sigma)
\bigg)^{\frac1p}.
$$

For a differentiable function $u$, the anisotropic radial derivative of $u$
with respect to $K$ is defined by
$$
\mathcal R_K(u)(x)
=
\nabla u(x)\cdot\sigma_K(x)
=
\nabla u(x)\cdot\frac{x}{\hK{x}},
\quad x\in\RN\setminus\{o\}.
$$
Equivalently, in anisotropic polar coordinates,
$$
\frac{d}{dr}u(r\sigma)
=
\nabla u(r\sigma)\cdot\sigma,
\quad r>0,\ \sigma\in\partial K.
$$
 When $K$ is the Euclidean unit ball, $\mathcal R_K(u)$ reduces
to the usual Euclidean radial derivative.

For later use, we also record the anisotropic Cauchy-Schwarz inequalities.
For $x,y\in\RN$, one has
\begin{align}\label{anisotropic-CS}
-\hK{x}\|-y\|_{K^*}
\le
x\cdot y
\le
\hK{x}\|y\|_{K^*}.
\end{align}
Indeed, if $x,y\ne o$, then
$
\sigma_K(x)\in\partial K,
\sigma_{K^*}(y):=\frac{y}{\|y\|_{K^*}}\in\partial K^*,
$
and hence
$$
\sigma_K(x)\cdot\sigma_{K^*}(y)\le1.
$$
This gives the right-hand side of \eqref{anisotropic-CS}; the left-hand side
follows by applying the same argument to $-y$.

For a set $E\subset\RN$, we denote by
$\operatorname{conv}(E)$ its convex hull. Define the
origin-symmetric convex body
$$
L
=
\operatorname{conv}\big(
K\cup(-K)
\big).
$$
Then
$$
L^*
=
K^*\cap(-K^*).
$$
Since $L^*$ is origin-symmetric, its Minkowski functional is a norm.
Moreover, the Minkowski functional of an intersection is the maximum
of the corresponding Minkowski functionals. Therefore,
\begin{align*}
\|y\|_{L^*}
=
\max\big\{
\|y\|_{K^*},
\|y\|_{-K^*}
\big\}
=
\max\big\{
\|y\|_{K^*},
\|-y\|_{K^*}
\big\}.
\end{align*}
Consequently, \eqref{anisotropic-CS} yields
\begin{align}
|x\cdot y|
\le
\|x\|_K\|y\|_{L^*},
\quad
x,y\in\RN.
\label{two-sided-anisotropic-CS}
\end{align}
If $K$ is origin-symmetric, then
$$
L=K,
\quad
L^*=K^*,
$$
and \eqref{two-sided-anisotropic-CS} reduces to the usual
anisotropic Cauchy-Schwarz inequality.

Let $\Ksc\subset\mathcal K_{(o)}$ be the class of convex bodies that are
smooth and strictly convex. If $K\in\Ksc$, then $K^*\in\Ksc$ and
$\hK{x}\in C^1(\RN\setminus\{o\})$. Moreover, for $x\in\RN\setminus\{o\}$, one has
$
x\cdot\nabla\hK{x}
=
\hK{x}$ and $
\big\|\nabla\hK{x}\big\|_{K^*}=1.
$
In particular, $\sigma_K(x)\in C^1(\RN\setminus\{o\})$.

We next recall the anisotropic symmetrization associated with $K$. This
will be used in Section \ref{s4}. For a measurable set $E\subset\RN$ with
$0<|E|<\infty$, where $|E|$ denotes its Lebesgue measure, the
$K$-symmetral of $E$ is defined by
$$
E^K=r_EK,
\quad
r_E=
\bigg(\frac{|E|}{|K|}\bigg)^{\frac1N}.
$$
Equivalently, up to a set of Lebesgue measure zero,
$$
E^K
=
\{x\in\RN:\hK{x}<r_E\}.
$$
Clearly, $|E^K|=|E|$.

Let $u$ be a measurable function on $\RN$ such that
$$
\mu_u(t):=|\{x\in\RN:|u(x)|>t\}|<\infty,
\quad t>0.
$$
The decreasing rearrangement of $u$ is defined by
$$
u^*(s)
=
\inf\{t>0:\mu_u(t)\le s\},
\quad s>0.
$$
The rearrangement of $u$ associated with $K$ is defined by
$$
u^K(x)
=
u^*\big(|K|\cdot\hK{x}^N\big),
\quad x\in\RN\setminus\{o\}.
$$
The value of $u^K$ at the origin is irrelevant and may be chosen arbitrarily.
By construction, $u^K$ is non-negative and radially symmetric with respect
to $K$.

The following elementary properties follow from the definition of $u^K$.
\begin{proposition}[\cite{VanSchaftingen}]\label{p:rearrangement-properties}
Let $K\in\mathcal K_{(o)}$, and  $u$ be a measurable function on
$\RN$ such that $\mu_u(t)<\infty$ for  $t>0$. Then the following
properties hold.
\begin{enumerate}
\item[(i)] $|u|$ and $u^K$ are equimeasurable, namely
$$
|\{x\in\RN:|u(x)|>t\}|
=
|\{x\in\RN:u^K(x)>t\}|
\quad\text{for }t>0.
$$

\item[(ii)] For $t>0$,
$$
\{x\in\RN:u^K(x)>t\}
=
\{x\in\RN:|u(x)|>t\}^K
$$
up to a set of Lebesgue measure zero.

\item[(iii)] For $q>0$,
$$
(|u|^q)^K=(u^K)^q
\quad\text{a.e. in }\RN.
$$

\item[(iv)] If $E\subset\RN$ is measurable and $0<|E|<\infty$, then
$$
(\chi_E)^K=\chi_{E^K}
\quad\text{a.e. in }\RN,
$$
where $\chi_E$ denotes the characteristic function of $E$.
\end{enumerate}
\end{proposition}

If, in addition, $K$ is origin-symmetric, then $\|\cdot\|_K$ is a norm
and the above construction reduces to the usual convex symmetrization
associated with $K$; see
\cite{AFTL97,Ferone-Volpicelli,VanSchaftingen}.

We use the following anisotropic P\'olya-Szeg\"{o} principle for convex
rearrangements.

\begin{lemma}[\cite{Ferone-Volpicelli,VanSchaftingen}]\label{l2}
Let $p\in(1,\infty)$, $K\in\mathcal K_{(o)}$, and
$u\in W^{1,p}(\RN)$. Then
\begin{align}\label{f54}
\int_{\RN}\|-\nabla u(x)\|_{K^*}^pdx
\ge
\int_{\RN}\|-\nabla u^K(x)\|_{K^*}^pdx.
\end{align}
Moreover, if $u\ge0$ and
$$
\big|
\{x:|\nabla u^K(x)|=0\}
\cap
\{x:0<u^K(x)<\operatorname*{ess\,sup}u\}
\big|
=0,
$$
where $\operatorname*{ess\,sup}u$ denotes the essential supremum of
$u$ on $\RN$, then equality holds in \eqref{f54} if and only if there
exists $x_0\in\RN$ such that
$$
u(x)=u^K(x+x_0)
\quad\text{a.e. on }\RN.
$$
\end{lemma}

We also need the following weighted form of the anisotropic Hardy-Littlewood
inequality for convex symmetrization.

\begin{proposition}[\cite{Wei}]\label{p1}
Let $p>0$, $K\in\mathcal K_{(o)}$, and  $u$ be a non-negative
measurable function on $\RN$.
\begin{enumerate}
\item[(i)] If $\omega:[0,\infty)\to[0,\infty)$ is non-increasing, then
\begin{align}\label{f39}
\int_{\RN}u(x)^p\omega(\hK{x})dx
\le
\int_{\RN}(u^K(x))^p\omega(\hK{x})dx.
\end{align}
If $\omega$ is strictly decreasing and equality holds in \eqref{f39}, then
$u=u^K$ a.e. on $\RN$.

\item[(ii)] If $v:[0,\infty)\to[0,\infty)$ is non-decreasing, then
\begin{align}\label{f35}
\int_{\RN}u(x)^pv(\hK{x})dx
\ge
\int_{\RN}(u^K(x))^pv(\hK{x})dx.
\end{align}
If $v$ is strictly increasing and equality holds in \eqref{f35}, then
$u=u^K$ a.e. on $\RN$.
\end{enumerate}
\end{proposition}

\section{Anisotropic \texorpdfstring{$L^p$}{Lp}-Hardy and \texorpdfstring{$L^p$}{Lp}-CKN identities on \texorpdfstring{$\RN$}{RN}}\label{s3}
In this section, we establish the anisotropic $L^p$-identity associated
with the Minkowski functional. This identity is the basic tool for the
sharp anisotropic $L^p$-CKN inequalities proved
in the next section. The main point is to replace the Euclidean radial
derivative by the anisotropic radial derivative.
Throughout this section, let 
$$
r=\|x\|_K,
\quad
\sigma_K(x)=\frac{x}{\|x\|_K},
\quad x\in \mathbb R^N\setminus\{o\}.
$$
Given $A,\psi\in C^1(0,\infty)$, we use the same symbols $A$ and
$\psi$ for the radial functions
$$
x\longmapsto A(\|x\|_K)
\quad\text{and}\quad
x\longmapsto\psi(\|x\|_K)
$$
on $\RN\setminus\{o\}$. Thus, when no confusion can arise, we write
$A$ and $\psi$ instead of $A(\|x\|_K)$ and
$\psi(\|x\|_K)$, respectively. Since $K\in\Ksc$ in the results
below, one has
$$
\|\cdot\|_K\in C^1(\RN\setminus\{o\}),
$$
and hence these radial functions belong to
$C^1(\RN\setminus\{o\})$.

Let $1<p<\infty$. For $s,t\in\mathbb R$, define
$$
R_p(s,t)
=
|t|^p+(p-1)|s|^p-p|s|^{p-2}st.
$$
By the convexity of the function $t\mapsto |t|^p$, one has
$
R_p(s,t)\ge 0
$
for all $s,t\in\mathbb R$. Moreover,
$
R_p(s,t)=0$
if and only if
$s=t.
$
When $p=2$, this remainder reduces to
$
R_2(s,t)=|t-s|^2.
$ Consequently, when $p=2$, the identities and inequalities below
reduce to the corresponding $L^2$ results in \cite{Wei}.
Thus $R_p$ is the natural nonlinear substitute for the square
completion used in the $L^2$ theory. Therefore, the following anisotropic $L^p$-identity holds.

\begin{theorem}\label{t1}
Let $1<p<\infty$, $\alpha>0$, $K\in\Ksc$, and
$A,\psi\in C^1(0,\infty)$. Assume, in addition, that
$
A|\psi|^{p-2}\psi\in C^1(0,\infty).
$
Let
$
u\in C_c^\infty(\mathbb R^N\setminus\{o\}).
$
Then
\begin{align*}
&\alpha^p
\int_{\mathbb R^N}
A|\mathcal R_K(u)|^pdx
+
\frac{p-1}{\alpha^{\frac{p}{p-1}}}
\int_{\mathbb R^N}
A|\psi|^p|u|^pdx\\
&=
-\int_{\mathbb R^N}
\operatorname{div}
\big(
A|\psi|^{p-2}\psi\sigma_K(x)
\big)
|u|^pdx
+
\int_{\mathbb R^N}
A
R_p\big(
\alpha^{-\frac{1}{p-1}}u\psi,
\alpha\mathcal R_K(u)
\big)dx .
\end{align*}
\end{theorem}
\begin{proof}
Since $u\in C_c^\infty(\mathbb R^N\setminus\{o\})$, integration by
parts gives
$$
-\int_{\mathbb R^N}\operatorname{div}(A|\psi|^{p-2}\psi\sigma_K)|u|^pdx
=
\int_{\mathbb R^N}A|\psi|^{p-2}\psi\sigma_K\cdot\nabla |u|^pdx=
p\int_{\mathbb R^N}
A|\psi|^{p-2}\psi|u|^{p-2}u\mathcal R_K(u)dx.
$$

On the other hand, one has
\begin{align*}
R_p\big(
\alpha^{-\frac{1}{p-1}}u\psi,
\alpha\mathcal R_K(u)
\big)
 =
\alpha^p|\mathcal R_K(u)|^p
+
(p-1)\alpha^{-\frac{p}{p-1}}|\psi|^p|u|^p
-
p|\psi|^{p-2}\psi|u|^{p-2}u\mathcal R_K(u).
\end{align*}
Multiplying this identity by $A$ and integrating over
$\mathbb R^N$, we get
\begin{align*}
&\int_{\mathbb R^N}
A
R_p\big(
\alpha^{-\frac{1}{p-1}}u\psi,
\alpha\mathcal R_K(u)
\big)dx
\\
&=
\alpha^p
\int_{\mathbb R^N}
A|\mathcal R_K(u)|^pdx
+
\frac{p-1}{\alpha^{\frac{p}{p-1}}}
\int_{\mathbb R^N}
A|\psi|^p|u|^pdx
-
p\int_{\mathbb R^N}
A|\psi|^{p-2}\psi|u|^{p-2}u\mathcal R_K(u)dx
\\
&=
\alpha^p
\int_{\mathbb R^N}
A|\mathcal R_K(u)|^pdx
+
\frac{p-1}{\alpha^{\frac{p}{p-1}}}
\int_{\mathbb R^N}
A|\psi|^p|u|^pdx
+
\int_{\mathbb R^N}
\operatorname{div}
\big(
A|\psi|^{p-2}\psi\sigma_K(x)
\big)
|u|^pdx.
\end{align*}
Rearranging the terms gives the desired identity.
\end{proof}

The additional regularity assumption $
A|\psi|^{p-2}\psi\in C^1(0,\infty)
$ in Theorem \ref{t1} is automatic for $p\ge2$
and is satisfied by all the choices of $A$ and $\psi$ used below.

\begin{lemma}\label{lem:radial-divergence}
Let $K\in\Ksc$ and let $f\in C^1(0,\infty)$. Then, with
$r=\|x\|_K$, one has
$$
\operatorname{div}\big(
f(r)\sigma_K(x)
\big)
=
f'(r)
+
\frac{N-1}{r}f(r),
\quad
x\in\RN\setminus\{o\}.
$$
\end{lemma}
\begin{proof}
Let
$
f(r)\sigma_K(x)
=
g(r)x,\
g(r)=\frac{f(r)}r.
$
Since
$
x\cdot\nabla\|x\|_K
=
\|x\|_K
=
r,
$
then
$$
\operatorname{div}\big(g(r)x\big)
=
Ng(r)
+
g'(r)x\cdot\nabla\|x\|_K=
Ng(r)+rg'(r).
$$
Since
$$
g'(r)
=
\frac{f'(r)}r
-
\frac{f(r)}{r^2},
$$
it follows that
\begin{align*}
Ng(r)+rg'(r)
=
\frac{Nf(r)}r
+
f'(r)
-
\frac{f(r)}r=
f'(r)
+
\frac{N-1}{r}f(r).
\end{align*}
This proves the result.
\end{proof}

\begin{remark}
The formula in Lemma \ref{lem:radial-divergence} coincides with the
usual Euclidean radial divergence formula. In particular, this
explains why the sharp constants in the radial inequalities below
depend only on $N$, $p$, $a$, and $b$, and not on the convex body
$K$.
\end{remark}

\subsection{Anisotropic \texorpdfstring{$L^p$}{Lp}-Hardy identities on \texorpdfstring{$\RN$}{RN}}
Taking $\alpha=1$ in Theorem \ref{t1} gives the following anisotropic $L^p$-Hardy-type identity and inequality.

\begin{theorem}\label{Hardy-identity}
Let $1<p<\infty$, $A,\psi\in C^1(0,\infty)$, $K\in\Ksc$, and
$u\in C_c^\infty(\mathbb R^N\setminus\{o\})$. 
Assume, in addition, that
$
A|\psi|^{p-2}\psi\in C^1(0,\infty).
$
Then
\begin{align*}
&\int_{\mathbb R^N}
A|\mathcal R_K(u)|^pdx
-
\int_{\mathbb R^N}
\bigg[
-\operatorname{div}
\big(
A|\psi|^{p-2}\psi\sigma_K(x)
\big)
-(p-1)A|\psi|^p
\bigg]
|u|^pdx\notag\\
&=
\int_{\mathbb R^N}
A R_p\big(u\psi,\mathcal R_K(u)\big)dx.
\end{align*}
In particular, if $A\ge0$, then
\begin{align*}
\int_{\mathbb R^N}
A|\mathcal R_K(u)|^pdx
\ge
\int_{\mathbb R^N}
\bigg[
-\operatorname{div}
\big(
A|\psi|^{p-2}\psi\sigma_K(x)
\big)
-(p-1)A|\psi|^p
\bigg]
|u|^pdx.
\end{align*}
\end{theorem}

We first derive several Hardy-type consequences of Theorem \ref{Hardy-identity}
by choosing suitable functions $A$ and $\psi$.

\begin{corollary}\label{Hardy-Heisenberg-cor} Let $1<p<\infty$, $K\in\Ksc$, and $p'=\frac{p}{p-1}. $ Then, for  $u\in C_c^\infty(\RN\setminus\{o\})$, one has \begin{align}\label{f7}
\int_{\RN}\!\!
|\mathcal R_K(u)|^pdx
\!=\!
N\int_{\RN}\!|u|^pdx
\!-\!
(p\!-\!1)\int_{\RN}\!
\hK{x}^{p'}|u|^pdx
\!+\!\!
\int_{\RN}\!\!
R_p\bigg(
\!\!-\!u\hK{x}^{\frac{1}{p-1}},
\mathcal R_K(u)
\bigg)dx.
\end{align}
\end{corollary} 
\begin{proof}
Taking $A=1$ and $\psi=-\hK{x}^{\frac{1}{p-1}}$ in Theorem \ref{Hardy-identity}.
Then
$
A|\psi|^{p-2}\psi
=
-r.
$
By Lemma \ref{lem:radial-divergence},
$$
\operatorname{div}
\big(
A|\psi|^{p-2}\psi\sigma_K(x)
\big)
=
\operatorname{div}\big(-r\sigma_K(x)\big)
=
-N.
$$
The conclusion follows from Theorem \ref{Hardy-identity}.
\end{proof}

\begin{lemma}\label{lem:Heisenberg-extension}
Let $1<p<\infty$, $p'=\frac{p}{p-1}$, and $K\in\Ksc$.
Suppose that $u\in W^{1,p}(\RN)$ and
$$
\int_{\RN}
\|x\|_K^{p'}|u|^pdx
<\infty.
$$
Then
\begin{align}
&\int_{\RN}
|\mathcal R_K(u)|^pdx
+
(p-1)
\int_{\RN}
\|x\|_K^{p'}|u|^pdx
-
N
\int_{\RN}
|u|^pdx
=
\int_{\RN}
R_p\big(
-u\|x\|_K^{\frac1{p-1}},
\mathcal R_K(u)
\big)dx.
\label{extended-Heisenberg-identity}
\end{align}
Consequently,
\begin{align}
\int_{\RN}
|\mathcal R_K(u)|^pdx
+
(p-1)
\int_{\RN}
\|x\|_K^{p'}|u|^pdx
\ge
N
\int_{\RN}
|u|^pdx.
\label{extended-Heisenberg-inequality}
\end{align}
Equality holds in \eqref{extended-Heisenberg-inequality} if and only if
$$
\mathcal R_K(u)
=
-\|x\|_K^{\frac1{p-1}}u
$$
almost everywhere in $\RN$.
\end{lemma}

\begin{proof}
Since $\sigma_K$ is bounded on $\RN\setminus\{o\}$, one has
$\mathcal R_K(u)\in L^p(\RN)$. Since
$
\nabla|u|^p
=
p|u|^{p-2}u\nabla u
$
almost everywhere, then,
\begin{align}
R_p\big(
-u\|x\|_K^{\frac1{p-1}},
\mathcal R_K(u)
\big)
=
|\mathcal R_K(u)|^p
+
(p-1)\|x\|_K^{p'}|u|^p
+
x\cdot\nabla|u|^p.
\label{pointwise-Heisenberg-identity}
\end{align}
Moreover, H\"older's inequality gives
\begin{align*}
\int_{\RN}
\big|x\cdot\nabla|u|^p\big|dx
&=
p
\int_{\RN}
\|x\|_K|u|^{p-1}|\mathcal R_K(u)|dx
\nonumber\\
&\le
p
\bigg(
\int_{\RN}
\|x\|_K^{p'}|u|^pdx
\bigg)^{\frac1{p'}}
\bigg(
\int_{\RN}
|\mathcal R_K(u)|^pdx
\bigg)^{\frac1p}
<\infty.
\end{align*}

Let $\eta\in C_c^\infty(\RN)$ satisfy
$$
0\le\eta\le1,\quad
\eta(x)=1\ \text{if }|x|\le1,\quad
\eta(x)=0\ \text{if }|x|\ge2,
$$
and let
$$
\eta_R(x)=\eta\bigg(\frac{x}{R}\bigg)\ \text{for}\ R>1.
$$
Since $|u|^p\in W^{1,1}(\RN)$, integration by parts yields
\begin{align*}
\int_{\RN}
\eta_Rx\cdot\nabla|u|^pdx
&=
-N
\int_{\RN}
\eta_R|u|^pdx
-
\int_{\RN}
x\cdot\nabla\eta_R|u|^pdx.
\end{align*}
The functions $x\cdot\nabla\eta_R$ are uniformly bounded and are
supported in
$
\{x\in\RN:R\le |x|\le2R\}.
$ Therefore,
$$
\int_{\RN}
x\cdot\nabla\eta_R|u|^pdx
\rightarrow0
$$
as $R\to\infty$. By dominated convergence,
$$
\int_{\RN}
x\cdot\nabla|u|^pdx
=
-N
\int_{\RN}
|u|^pdx.
$$
Integrating \eqref{pointwise-Heisenberg-identity} proves
\eqref{extended-Heisenberg-identity}.

Since $R_p(s,t)\ge0$, inequality
\eqref{extended-Heisenberg-inequality} follows. Finally,
$$
R_p(s,t)=0
\quad\Longleftrightarrow\quad
s=t,
$$
which proves the equality statement.
\end{proof}

\begin{corollary}\label{weighted-Hardy-identity} Let $1<p<\infty$, $\lambda\in\mathbb R$, and $K\in\Ksc$. Then, for  $u\in C_c^\infty(\RN\setminus\{o\})$, one has \begin{align*}
\int_{\mathbb R^N}
\frac{|\mathcal R_K(u)|^p}{\|x\|_K^\lambda}dx
-\!
\bigg|\frac{N-\lambda-p}{p}\bigg|^p
\int_{\mathbb R^N}
\frac{|u|^p}{\|x\|_K^{\lambda+p}}dx
=\!
\int_{\mathbb R^N}
\frac{1}{\|x\|_K^\lambda}
R_p\bigg(
-\frac{N-\lambda-p}{p}
\frac{u}{\|x\|_K},
\mathcal R_K(u)
\bigg)dx.
\end{align*} 
\end{corollary}
\begin{proof}
Let
$
c
=
\frac{N-\lambda-p}{p}.
$
Taking
$
A=r^{-\lambda},
\
\psi=-\frac{c}{r}.
$
Then
\begin{align*}
A|\psi|^{p-2}\psi
=
-r^{-\lambda}
\bigg|\frac{c}{r}\bigg|^{p-2}
\frac{c}{r}=
-|c|^{p-2}c
r^{1-\lambda-p}.
\end{align*}
By Lemma \ref{lem:radial-divergence},
\begin{align*}
\operatorname{div}
\big(
A|\psi|^{p-2}\psi\sigma_K(x)
\big)
=
-|c|^{p-2}c
\big(
N-\lambda-p
\big)
r^{-\lambda-p}=
-p|c|^p r^{-\lambda-p}.
\end{align*}
The result follows from Theorem \ref{Hardy-identity}.
\end{proof}

\begin{corollary}\label{c3}
Let $1<p<\infty$, $K\in\Ksc$, $a,b\in\mathbb R$, and
$u\in C_c^\infty(\mathbb R^N\setminus\{o\})$.
Then
\begin{align}\label{f16}
&\int_{\mathbb R^N}
\frac{|\mathcal R_K(u)|^p}{\|x\|_K^{bp}}dx
+
(p-1)
\int_{\mathbb R^N}
\frac{|u|^p}{\|x\|_K^{ap}}dx
-
\big(N-1-(p-1)a-b\big)
\int_{\mathbb R^N}
\frac{|u|^p}{\|x\|_K^{(p-1)a+b+1}}dx
\nonumber
\\
&=
\int_{\mathbb R^N}
\frac{1}{\|x\|_K^{bp}}
R_p\big(
u\|x\|_K^{b-a},
-\mathcal R_K(u)
\big)dx.
\end{align}
If
$
N-1-(p-1)a-b>0,
$
then, since $R_p\ge0$, one obtains 
\begin{align}\label{f4-1}
\int_{\mathbb R^N}
\frac{|\mathcal R_K(u)|^p}{\|x\|_K^{bp}}dx
+
(p-1)
\int_{\mathbb R^N}
\frac{|u|^p}{\|x\|_K^{ap}}dx
\ge
\big(N-1-(p-1)a-b\big)
\int_{\mathbb R^N}
\frac{|u|^p}{\|x\|_K^{(p-1)a+b+1}}dx.
\end{align}
In particular, if
$
b+1-a>0,
b\le\frac{N-p}{p},
$
then
\begin{align*}
N-1-(p-1)a-b
=
N-p(b+1)+(p-1)(b+1-a)>0.
\end{align*}
\end{corollary}
\begin{proof}
Taking
$
A=r^{-bp},
\
\psi=-r^{b-a}.
$
Then
$
A|\psi|^{p-2}\psi
=
-r^{-bp+(b-a)(p-1)}=
-r^{-(p-1)a-b}.
$
By Lemma \ref{lem:radial-divergence},
\begin{align*}
\operatorname{div}
\big(
A|\psi|^{p-2}\psi\sigma_K(x)
\big)
=
-\big(
N-1-(p-1)a-b
\big)
r^{-(p-1)a-b-1}.
\end{align*}
Therefore, Theorem \ref{Hardy-identity} gives
\eqref{f16}. Since
$
R_p(s,t)\ge0,
$
inequality \eqref{f4-1} follows.
\end{proof}

\begin{corollary}\label{c4}
Let $1<p<\infty$, $K\in\Ksc$, $a,b\in\mathbb R$, and
$u\in C_c^\infty(\mathbb R^N\setminus\{o\})$.
Then
\begin{align}\label{f17}
&\int_{\mathbb R^N}
\frac{|\mathcal R_K(u)|^p}{\|x\|_K^{bp}}dx
+
(p-1)
\int_{\mathbb R^N}
\frac{|u|^p}{\|x\|_K^{ap}}dx
-
\big((p-1)a+b+1-N\big)
\int_{\mathbb R^N}
\frac{|u|^p}{\|x\|_K^{(p-1)a+b+1}}dx
\nonumber
\\
&=
\int_{\mathbb R^N}
\frac{1}{\|x\|_K^{bp}}
R_p\big(
u\|x\|_K^{b-a},
\mathcal R_K(u)
\big)dx.
\end{align}
If
$
(p-1)a+b+1-N>0,
$
then, since $R_p\ge0$, one obtains
\begin{align}\label{f4-2}
\int_{\mathbb R^N}
\frac{|\mathcal R_K(u)|^p}{\|x\|_K^{bp}}dx
+
(p-1)
\int_{\mathbb R^N}
\frac{|u|^p}{\|x\|_K^{ap}}dx
\ge
\big((p-1)a+b+1-N\big)
\int_{\mathbb R^N}
\frac{|u|^p}{\|x\|_K^{(p-1)a+b+1}}dx.
\end{align}
In particular, if
$
b+1-a<0,
b\ge\frac{N-p}{p},
$
then
\begin{align*}
(p-1)a+b+1-N
=
p(b+1)-N+(p-1)(a-b-1)
>0.
\end{align*}
\end{corollary}
\begin{proof}
Taking
$
A=r^{-bp},
\
\psi=r^{b-a}.
$
Then
$
A|\psi|^{p-2}\psi
=
r^{-bp+(b-a)(p-1)}=
r^{-(p-1)a-b}.
$
By Lemma \ref{lem:radial-divergence},
\begin{align*}
\operatorname{div}
\big(
A|\psi|^{p-2}\psi\sigma_K(x)
\big)
=
-\big(
1+(p-1)a+b-N
\big)
r^{-(p-1)a-b-1}.
\end{align*}
Therefore, Theorem \ref{Hardy-identity} gives
\eqref{f17}. Since
$
R_p(s,t)\ge0,
$
inequality \eqref{f4-2} follows.
\end{proof}

\subsection{Anisotropic \texorpdfstring{$L^p$}{Lp}-CKN identities on \texorpdfstring{$\RN$}{RN}}
We now optimize $\alpha$ in Theorem \ref{t1}, i.e., choosing $
\alpha
=
\Big(
\frac{
\int_{\mathbb R^N}
A(r)|\psi(r)|^p|u|^pdx
}{
\int_{\mathbb R^N}
A(r)|\mathcal R_K(u)|^pdx
}
\Big)^{\frac{p-1}{p^2}}
$. This gives the
anisotropic $L^p$-CKN identity and inequality.

\begin{theorem}\label{CKN-identity}
Let $1<p<\infty$, $K\in\Ksc$, $A,\psi\in C^1(0,\infty)$ with $A\ge0$, and
$
A|\psi|^{p-2}\psi\in C^1(0,\infty).
$
Let
$
u\in C_c^\infty(\mathbb R^N\setminus\{o\})\setminus\{0\}.
$
Suppose that
$$
\int_{\mathbb R^N}A|\mathcal R_K(u)|^pdx>0\ \ \mathrm{and}\
\int_{\mathbb R^N}A|\psi|^p|u|^pdx>0.
$$
Then
\begin{align*}
&\bigg(
\int_{\mathbb R^N}
A|\mathcal R_K(u)|^pdx
\bigg)^{\frac1p}
\bigg(
\int_{\mathbb R^N}
A|\psi|^p|u|^pdx
\bigg)^{\frac{p-1}{p}}
+
\frac{1}{p}
\int_{\mathbb R^N}
\operatorname{div}
\big(
A|\psi|^{p-2}\psi\sigma_K(x)
\big)
|u|^pdx
\\
&=
\frac{1}{p}
\int_{\mathbb R^N}
A
R_p
\Bigg(
\bigg(
\frac{
\int_{\mathbb R^N}A|\mathcal R_K(u)|^pdx
}{
\int_{\mathbb R^N}A|\psi|^p|u|^pdx
}
\bigg)^{\frac{1}{p^2}}
u\psi,
\bigg(
\frac{
\int_{\mathbb R^N}A|\psi|^p|u|^pdx
}{
\int_{\mathbb R^N}A|\mathcal R_K(u)|^pdx
}
\bigg)^{\frac{p-1}{p^2}}
\mathcal R_K(u)
\Bigg)dx.
\end{align*}
In particular, one obtains
\begin{align*}
\bigg(
\int_{\mathbb R^N}
A|\mathcal R_K(u)|^pdx
\bigg)^{\frac1p}
\bigg(
\int_{\mathbb R^N}
A|\psi|^p|u|^pdx
\bigg)^{\frac{p-1}{p}}
\ge
-\frac{1}{p}
\int_{\mathbb R^N}
\operatorname{div}
\big(
A|\psi|^{p-2}\psi\sigma_K(x)
\big)
|u|^pdx .
\end{align*}
\end{theorem}
We now derive several consequences of Theorem \ref{CKN-identity} by choosing
suitable functions $A$ and $\psi$.

\begin{corollary}\label{cor:Heisenberg-p}
Let $1<p<\infty$, 
$
p'=\frac{p}{p-1},
$
and $K\in\Ksc$. Then, for 
$
u\in C_c^\infty(\mathbb R^N\setminus\{o\})\setminus\{0\},
$
one has
\begin{align}\label{f8}
&\bigg(
\int_{\RN}
|\mathcal R_K(u)|^pdx
\bigg)^{\frac1p}
\bigg(
\int_{\RN}
\hK{x}^{p'}|u|^pdx
\bigg)^{\frac1{p'}}
-
\frac{N}{p}
\int_{\RN}|u|^pdx
\nonumber
\\
&=
\frac1p
\int_{\RN}
R_p\Bigg(
-
\bigg(
\frac{
\int_{\RN}|\mathcal R_K(u)|^pdx
}{
\int_{\RN}\hK{x}^{p'}|u|^pdx
}
\bigg)^{\frac{1}{p^2}}
u\hK{x}^{\frac{1}{p-1}},
\bigg(
\frac{
\int_{\RN}\hK{x}^{p'}|u|^pdx
}{
\int_{\RN}|\mathcal R_K(u)|^pdx
}
\bigg)^{\frac{p-1}{p^2}}
\mathcal R_K(u)
\Bigg)dx.
\end{align}
\end{corollary}
\begin{proof}
Taking
$
A=1,
\
\psi=-r^{\frac1{p-1}}.
$
As in the proof of Corollary \ref{Hardy-Heisenberg-cor},
Lemma \ref{lem:radial-divergence} gives
$$
\operatorname{div}
\big(
A|\psi|^{p-2}\psi\sigma_K(x)
\big)
=
-N.
$$
The conclusion follows from Theorem \ref{CKN-identity}.
\end{proof}

\begin{corollary}
Let $1<p<\infty$,  $K\in\Ksc$, $a,b\in\mathbb R$, and
$u\in C_c^\infty(\mathbb R^N\setminus\{o\})$.
Then
\begin{align}\label{f0-1}
&\bigg(
\int_{\mathbb R^N}
\frac{|\mathcal R_K(u)|^p}{\|x\|_K^{bp}}dx
\bigg)^{\frac{1}{p}}
\bigg(
\int_{\mathbb R^N}
\frac{|u|^p}{\|x\|_K^{ap}}dx
\bigg)^{\frac{p-1}{p}}
-
\frac{N-1-(p-1)a-b}{p}
\int_{\mathbb R^N}
\frac{|u|^p}{\|x\|_K^{(p-1)a+b+1}}dx
\nonumber
\\
&=
\frac{1}{p}
\int_{\mathbb R^N}
\frac{1}{\|x\|_K^{bp}}
R_p
\Bigg(
-
\Bigg(
\frac{
\int_{\mathbb R^N}
\frac{|\mathcal R_K(u)|^p}{\|x\|_K^{bp}}dx
}{
\int_{\mathbb R^N}
\frac{|u|^p}{\|x\|_K^{ap}}dx
}
\Bigg)^{\frac{1}{p^2}}
u\|x\|_K^{b-a},
\Bigg(
\frac{
\int_{\mathbb R^N}
\frac{|u|^p}{\|x\|_K^{ap}}dx
}{
\int_{\mathbb R^N}
\frac{|\mathcal R_K(u)|^p}{\|x\|_K^{bp}}dx
}
\Bigg)^{\frac{p-1}{p^2}}
\mathcal R_K(u)
\Bigg)dx.
\end{align}
If
$
N-1-(p-1)a-b>0,
$
then, since $R_p\ge0$, one obtains 
\begin{align}\label{f0-2}
&\bigg(
\int_{\mathbb R^N}
\frac{|\mathcal R_K(u)|^p}{\|x\|_K^{bp}}dx
\bigg)^{\frac{1}{p}}
\bigg(
\int_{\mathbb R^N}
\frac{|u|^p}{\|x\|_K^{ap}}dx
\bigg)^{\frac{p-1}{p}}
\ge
\frac{N-1-(p-1)a-b}{p}
\int_{\mathbb R^N}
\frac{|u|^p}{\|x\|_K^{(p-1)a+b+1}}dx.
\end{align}
In particular, if
$
b+1-a>0,
b\le\frac{N-p}{p},
$
then
\begin{align*}
N-1-(p-1)a-b
=
N-p(b+1)+(p-1)(b+1-a)>0.
\end{align*}
\end{corollary}
\begin{proof}
Taking
$
A=r^{-bp},
\
\psi=-r^{b-a}.
$
As in the proof of Corollary \ref{c3},
Lemma \ref{lem:radial-divergence} gives
$$
\operatorname{div}
\big(
A|\psi|^{p-2}\psi\sigma_K(x)
\big)
=
-\big(
N-1-(p-1)a-b
\big)
r^{-(p-1)a-b-1}.
$$
Therefore, Theorem \ref{CKN-identity} gives
\eqref{f0-1}. Since
$
R_p(s,t)\ge0,
$
inequality \eqref{f0-2} follows.
\end{proof}

\begin{corollary}
Let $1<p<\infty$, let $K\in\Ksc$, $a,b\in\mathbb R$, and
$u\in C_c^\infty(\mathbb R^N\setminus\{o\})$.
Then
\begin{align}\label{f1-1}
&\bigg(
\int_{\mathbb R^N}
\frac{|\mathcal R_K(u)|^p}{\|x\|_K^{bp}}dx
\bigg)^{\frac{1}{p}}
\bigg(
\int_{\mathbb R^N}
\frac{|u|^p}{\|x\|_K^{ap}}dx
\bigg)^{\frac{p-1}{p}}
-
\frac{(p-1)a+b+1-N}{p}
\int_{\mathbb R^N}
\frac{|u|^p}{\|x\|_K^{(p-1)a+b+1}}dx
\nonumber
\\
&=
\frac{1}{p}
\int_{\mathbb R^N}
\frac{1}{\|x\|_K^{bp}}
R_p
\Bigg(
\Bigg(
\frac{
\int_{\mathbb R^N}
\frac{|\mathcal R_K(u)|^p}{\|x\|_K^{bp}}dx
}{
\int_{\mathbb R^N}
\frac{|u|^p}{\|x\|_K^{ap}}dx
}
\Bigg)^{\frac{1}{p^2}}
u\|x\|_K^{b-a},
\Bigg(
\frac{
\int_{\mathbb R^N}
\frac{|u|^p}{\|x\|_K^{ap}}dx
}{
\int_{\mathbb R^N}
\frac{|\mathcal R_K(u)|^p}{\|x\|_K^{bp}}dx
}
\Bigg)^{\frac{p-1}{p^2}}
\mathcal R_K(u)
\Bigg)dx.
\end{align}
If
$
(p-1)a+b+1-N>0,
$
then, since $R_p\ge0$, one obtains
\begin{align}\label{f1-2}
&\bigg(
\int_{\mathbb R^N}
\frac{|\mathcal R_K(u)|^p}{\|x\|_K^{bp}}dx
\bigg)^{\frac{1}{p}}
\bigg(
\int_{\mathbb R^N}
\frac{|u|^p}{\|x\|_K^{ap}}dx
\bigg)^{\frac{p-1}{p}}
\ge
\frac{(p-1)a+b+1-N}{p}
\int_{\mathbb R^N}
\frac{|u|^p}{\|x\|_K^{(p-1)a+b+1}}dx.
\end{align}
In particular, if
$
b+1-a<0,
b\ge\frac{N-p}{p},
$
then
\begin{align*}
(p-1)a+b+1-N
=
p(b+1)-N+(p-1)(a-b-1)
>0.
\end{align*}
\end{corollary}
\begin{proof}
Taking
$
A=r^{-bp},
\
\psi=r^{b-a}.
$
As in the proof of Corollary \ref{c4},
Lemma \ref{lem:radial-divergence} gives
$$
\operatorname{div}
\big(
A|\psi|^{p-2}\psi\sigma_K(x)
\big)
=
-\big(
1+(p-1)a+b-N
\big)
r^{-(p-1)a-b-1}.
$$
Therefore, Theorem \ref{CKN-identity} gives
\eqref{f1-1}. Since
$
R_p(s,t)\ge0,
$
inequality \eqref{f1-2} follows.
\end{proof}
\begin{remark}
The identities \eqref{f16}, \eqref{f17}, \eqref{f0-1}, and
\eqref{f1-1} are valid for all $a,b\in\mathbb R$.
The parameter regions used above are imposed only to guarantee the
positivity of the constants in the corresponding inequalities.
These regions are sufficient, but not necessary, for positivity.
\end{remark}

\section{Sharp anisotropic \texorpdfstring{$L^p$}{Lp}-CKN inequality}\label{s4}
In this section, we apply the anisotropic $L^p$-identities obtained in
Section \ref{s3} to prove sharp inequalities associated with the Minkowski
functional. We first discuss the gradient anisotropic $L^p$-Heisenberg uncertainty principle. Its proof combines the radial identity with the anisotropic P\'olya-Szeg\"{o} principle and the weighted anisotropic 
Hardy-Littlewood inequality. We then prove the sharp anisotropic radial $L^p$-CKN inequalities and characterize the corresponding extremal functions in the natural completion space. As a special case of the $L^p$-CKN inequality, we obtain the anisotropic radial $L^p$-Heisenberg uncertainty principle. 
Finally, using the two-sided anisotropic Cauchy-Schwarz inequality
associated with
$L=\operatorname{conv}\big(K\cup(-K)\big)$, we derive norm-based
anisotropic gradient versions of the radial inequalities.

\subsection{Sharp constants and optimizers for
the anisotropic \texorpdfstring{$L^p$}{Lp}-Heisenberg Uncertainty Principle}
This subsection is devoted to the sharp anisotropic
$L^p$-Heisenberg uncertainty principle associated with the Minkowski
functional. The proof combines Lemma
\ref{lem:Heisenberg-extension} with the anisotropic
P\'olya-Szeg\"o principle and the weighted anisotropic
Hardy-Littlewood inequality. For  $1<p<\infty$ and $p'=\frac{p}{p-1}$, let
$$
\mathcal C
=
W^{1,p}(\RN)
\cap
\bigg\{
u:
\int_{\RN}
\|x\|_K^{p'}|u|^pdx
<\infty
\bigg\}.
$$

\begin{theorem}\label{thm:gradient-Heisenberg}
Let $1<p<\infty$, $p'=\frac{p}{p-1}$, and $K\in\Ksc$.
Then, for every nonnegative $u\in\mathcal C$, the following
statements hold.

\begin{itemize}

\item[(i)]
\begin{align}
\int_{\RN}
\|-\nabla u\|_{K^*}^pdx
+
(p-1)
\int_{\RN}
\|x\|_K^{p'}u^pdx
\ge
N
\int_{\RN}
u^pdx.
\label{Lp-Heisenberg-additive}
\end{align}
The constant $N$ is sharp. If $u\in\mathcal C\backslash\{0\}$, equality holds in
\eqref{Lp-Heisenberg-additive} if and only if
$$
u(x)
=
\kappa
\exp\bigg(
-\frac1{p'}\|x\|_K^{p'}
\bigg),
\quad
\kappa>0.
$$

\item[(ii)]
\begin{align}
\bigg(
\int_{\RN}
\|-\nabla u\|_{K^*}^pdx
\bigg)^{\frac1p}
\bigg(
\int_{\RN}
\|x\|_K^{p'}u^pdx
\bigg)^{\frac1{p'}}
\ge
\frac Np
\int_{\RN}
u^pdx.
\label{Lp-Heisenberg-product}
\end{align}
The constant $\frac{N}{p}$ is sharp. If $u\in\mathcal C\backslash\{0\}$, equality holds in
\eqref{Lp-Heisenberg-product} if and only if
$$
u(x)
=
\beta
\exp\bigg(
-\frac1{p'\lambda^{p'}}
\|x\|_K^{p'}
\bigg),
\quad
\beta>0,\quad\lambda>0.
$$

\end{itemize}
\end{theorem}

\begin{proof}
We first prove part $(i)$. By the anisotropic P\'olya-Szeg\"o
principle,
\begin{align}
\int_{\RN}
\|-\nabla u\|_{K^*}^pdx
\ge
\int_{\RN}
\|-\nabla u^K\|_{K^*}^pdx.
\label{Heisenberg-PS}
\end{align}
Let $r=\|x\|_K$. Since 
$
 r^{p'}
$
is strictly increasing, Proposition \ref{p1} gives
\begin{align}
\int_{\RN}
\|x\|_K^{p'}u^pdx
\ge
\int_{\RN}
\|x\|_K^{p'}(u^K)^pdx.
\label{Heisenberg-HL}
\end{align}
The function $u^K$ is nonnegative, nonincreasing, and radially
symmetric with respect to $K$. Hence,
$$
\|-\nabla u^K\|_{K^*}
=
|\mathcal R_K(u^K)|
$$
almost everywhere. 
By the anisotropic P\'olya-Szeg\"o principle together with the
equimeasurability of $u$ and $u^K$, one has
$
u^K\in W^{1,p}(\mathbb R^N).
$
Moreover, by \eqref{Heisenberg-HL},
$$
\int_{\mathbb R^N}
\|x\|_K^{p'}(u^K)^pdx
<\infty.
$$
Therefore, Lemma \ref{lem:Heisenberg-extension} applies to $u^K$.
By Lemma \ref{lem:Heisenberg-extension}, one has
\begin{align}
&\int_{\RN}
|\mathcal R_K(u^K)|^pdx
+
(p-1)
\int_{\RN}
\|x\|_K^{p'}(u^K)^pdx
\ge
N
\int_{\RN}
(u^K)^pdx.
\label{Heisenberg-radial-extension}
\end{align}
Combining \eqref{Heisenberg-PS},
\eqref{Heisenberg-HL}, and
\eqref{Heisenberg-radial-extension}, and using the
equimeasurability of $u$ and $u^K$, we obtain
\begin{align*}
&\int_{\RN}
\|-\nabla u\|_{K^*}^pdx
+
(p-1)
\int_{\RN}
\|x\|_K^{p'}u^pdx
\nonumber\\
&\ge
\int_{\RN}
|\mathcal R_K(u^K)|^pdx
+
(p-1)
\int_{\RN}
\|x\|_K^{p'}(u^K)^pdx
\nonumber\\
&\ge
N
\int_{\RN}
(u^K)^pdx
=
N
\int_{\RN}
u^pdx.
\end{align*}
This proves \eqref{Lp-Heisenberg-additive}.

Suppose that $u\not\equiv0$ and equality holds in
\eqref{Lp-Heisenberg-additive}. Then equality must hold in
\eqref{Heisenberg-HL}. Since $r^{p'}$ is strictly increasing,
Proposition \ref{p1} implies
$
u=u^K
$
almost everywhere in $\RN$. Equality must also hold in
\eqref{Heisenberg-radial-extension}. Lemma
\ref{lem:Heisenberg-extension} therefore gives
$$
\mathcal R_K(u)
=
-\|x\|_K^{\frac1{p-1}}u
$$
almost everywhere.

Since $u=u^K$, there exists a nonincreasing function, still denoted
by $u$, such that
$$
u(x)=u(r),
\quad
r=\|x\|_K.
$$
By the anisotropic polar formula and Fubini's theorem, one has
$
u\in W_{\mathrm{loc}}^{1,p}(0,\infty)
$
and
$
u'(r)
=
-r^{\frac1{p-1}}u(r)
$
for almost every $r>0$. Therefore,
$$
u(r)
=
\kappa
\exp\bigg(
-\frac1{p'}r^{p'}
\bigg)
$$
for some $\kappa>0$. 
Conversely, this function satisfies
$$
\mathcal R_K(u)
=
-\|x\|_K^{\frac1{p-1}}u,
$$
and all the inequalities above become equalities. This proves part
$(i)$, including the sharpness of the constant.

We next prove part $(ii)$. For $\tau>0$, define
$$
u_\tau(x)=u(\tau x).
$$
Applying part $(i)$ to $u_\tau$ and making the change of variables
$y=\tau x$, we obtain
\begin{align*}
&\tau^p
\int_{\RN}
\|-\nabla u\|_{K^*}^pdx
+
(p-1)\tau^{-p'}
\int_{\RN}
\|x\|_K^{p'}u^pdx
\ge
N
\int_{\RN}
u^pdx.
\end{align*}
Minimizing the left-hand side over $\tau>0$ gives
\begin{align*}
&p
\bigg(
\int_{\RN}
\|-\nabla u\|_{K^*}^pdx
\bigg)^{\frac1p}
\bigg(
\int_{\RN}
\|x\|_K^{p'}u^pdx
\bigg)^{\frac1{p'}}
\ge
N
\int_{\RN}
u^pdx,
\end{align*}
which proves \eqref{Lp-Heisenberg-product}.

Suppose that $u\not\equiv0$ and equality holds in
\eqref{Lp-Heisenberg-product}. Let
$$
\tau_0
=
\Bigg(
\frac{
\int_{\RN}
\|x\|_K^{p'}u^pdx
}{
\int_{\RN}
\|-\nabla u\|_{K^*}^pdx
}
\Bigg)^{\frac1{pp'}}.
$$
Then
\begin{align*}
&\tau_0^p
\int_{\RN}
\|-\nabla u\|_{K^*}^pdx
+
(p-1)\tau_0^{-p'}
\int_{\RN}
\|x\|_K^{p'}u^pdx
\nonumber\\
&=
p
\bigg(
\int_{\RN}
\|-\nabla u\|_{K^*}^pdx
\bigg)^{\frac1p}
\bigg(
\int_{\RN}
\|x\|_K^{p'}u^pdx
\bigg)^{\frac1{p'}}
\nonumber\\
&=
N
\int_{\RN}
u^pdx.
\end{align*}
Thus, $u_{\tau_0}$ attains equality in
\eqref{Lp-Heisenberg-additive}. By part $(i)$,
$$
u_{\tau_0}(x)
=
\beta
\exp\bigg(
-\frac1{p'}\|x\|_K^{p'}
\bigg)
$$
for some $\beta>0$. Since
$
u_{\tau_0}(x)=u(\tau_0x),
$
it follows that
$$
u(x)
=
\beta
\exp\bigg(
-\frac1{p'\tau_0^{p'}}
\|x\|_K^{p'}
\bigg).
$$
Writing $\lambda=\tau_0$ gives the stated family of extremal
functions.

Conversely, suppose that
$$
u(x)
=
\beta
\exp\bigg(
-\frac1{p'\lambda^{p'}}
\|x\|_K^{p'}
\bigg),
\quad
\beta>0,\quad\lambda>0.
$$
Then
$$
u_\lambda(x)
=
\beta
\exp\bigg(
-\frac1{p'}\|x\|_K^{p'}
\bigg),
$$
so $u_\lambda$ attains equality in
\eqref{Lp-Heisenberg-additive}. Reversing the scaling argument shows
that equality holds in \eqref{Lp-Heisenberg-product}. This completes
the proof.
\end{proof}

\begin{remark}
The zero function also attains equality in both inequalities in
Theorem \ref{thm:gradient-Heisenberg}. Therefore, the complete
equality classes consist of $u\equiv0$ together with the corresponding
nonzero extremal families stated in Theorem \ref{thm:gradient-Heisenberg}.
\end{remark}

\begin{remark}
For a general, possibly sign-changing, function $u\in\mathcal C$,
applying Theorem \ref{thm:gradient-Heisenberg} to $|u|$ and using
$
\|\pm\nabla u\|_{K^*}
\le
\|\nabla u\|_{L^*}
$
almost everywhere, we obtain
\begin{align*}
\int_{\mathbb R^N}
\|\nabla u\|_{L^*}^pdx
+
(p-1)
\int_{\mathbb R^N}
\|x\|_K^{p'}|u|^pdx
\ge
N
\int_{\mathbb R^N}|u|^pdx.
\end{align*}
Moreover,
\begin{align*}
\left(
\int_{\mathbb R^N}
\|\nabla u\|_{L^*}^pdx
\right)^{\frac1p}
\left(
\int_{\mathbb R^N}
\|x\|_K^{p'}|u|^pdx
\right)^{\frac{1}{p'}}
\ge
\frac{N}{p}
\int_{\mathbb R^N}|u|^pdx.
\end{align*}
\end{remark}

\subsection{Sharp constants and optimizers for
the anisotropic \texorpdfstring{$L^p$}{Lp}-CKN inequality}
In this section, we prove the sharp anisotropic 
$L^p$-CKN inequality associated with the
Minkowski functional.
Let
$\mathcal C_{K,a,b}^p$ be the completion of
$C_c^\infty(\mathbb R^N\setminus\{o\})$ with respect to the norm
$$
\|u\|_{\mathcal C_{K,a,b}^p}
=
\bigg(
\int_{\mathbb R^N}
\frac{|\mathcal R_K(u)|^p}{\|x\|_K^{bp}}dx
+
\int_{\mathbb R^N}
\frac{|u|^p}{\|x\|_K^{ap}}dx
\bigg)^{\frac1p}.
$$

\begin{lemma}\label{lem:completion-representatives}
Let $1<p<\infty$, $K\in\Ksc$, and $a,b\in\R$.
Every element of $\mathcal C_{K,a,b}^p$ admits a unique
representative
$
u\in L_{\mathrm{loc}}^p(\RN\setminus\{o\}).
$
Moreover, the anisotropic radial derivative extends uniquely and
continuously to $\mathcal C_{K,a,b}^p$, and
$$
\frac{\mathcal R_K(u)}{\|x\|_K^b}\in L^p(\RN),
\quad
\frac{u}{\|x\|_K^a}\in L^p(\RN).
$$
For $\mu_K$-almost every $\sigma\in\partial K$, the function
$
r\mapsto u(r\sigma)
$
belongs to $W_{\mathrm{loc}}^{1,p}(0,\infty)$ and hence is locally
absolutely continuous. Furthermore,
$$
\frac{d}{dr}u(r\sigma)
=
\mathcal R_K(u)(r\sigma)
$$
for almost every $r>0$ and for $\mu_K$-almost every
$\sigma\in\partial K$.
\end{lemma}

\begin{proof}
Let $\{u_n\}\subset C_c^\infty(\RN\setminus\{o\})$ be a Cauchy
sequence representing an element of $\mathcal C_{K,a,b}^p$. By the
completeness of the weighted $L^p$-spaces, there exist measurable
functions $u$ and $v$ such that
\begin{align}
\int_{\RN}
\frac{|u_n-u|^p}{\|x\|_K^{ap}}dx
&\rightarrow0\ \text{as}\ n\to\infty,
\label{completion-u-convergence}\\
\int_{\RN}
\frac{|\mathcal R_K(u_n)-v|^p}{\|x\|_K^{bp}}dx
&\rightarrow0\ \text{as}\ n\to\infty.
\label{completion-R-convergence}
\end{align}

Fix $0<\rho<R<\infty$ and let
$
\Omega_{\rho,R}
=
\{x\in\RN:\rho<\|x\|_K<R\}.
$
Since the weights $\|x\|_K^{-ap}$ and $\|x\|_K^{-bp}$ are bounded
above and below by positive constants on $\Omega_{\rho,R}$,
\eqref{completion-u-convergence} and
\eqref{completion-R-convergence} imply
\begin{align}
u_n&\rightarrow u
\quad\text{in }L^p(\Omega_{\rho,R})\ \text{as}\ n\to\infty,\label{f-u}\\
\mathcal R_K(u_n)&\rightarrow v
\quad\text{in }L^p(\Omega_{\rho,R})\ \text{as}\ n\to\infty.\label{f-Ru}
\end{align}
Thus, $u\in L_{\mathrm{loc}}^p(\RN\setminus\{o\})$. This
representative is independent of the approximating sequence, since
two representatives arising from equivalent Cauchy sequences agree
on every annulus $\Omega_{\rho,R}$.

For every $n$, one has
$$
\frac{d}{dr}u_n(r\sigma)
=
\mathcal R_K(u_n)(r\sigma).
$$
Since $r^{N-1}$ is bounded above and below by positive constants on
$[\rho,R]$, the anisotropic polar formula together with
\eqref{f-u} and \eqref{f-Ru} yields
\begin{align*}
&\int_{\partial K}\int_\rho^R
|u_n(r\sigma)-u(r\sigma)|^pdrd\mu_K(\sigma)
\rightarrow0\ \text{as}\ n\to\infty,\\
&\int_{\partial K}\int_\rho^R
|\mathcal R_K(u_n)(r\sigma)-v(r\sigma)|^p
drd\mu_K(\sigma)
\rightarrow0\ \text{as}\ n\to\infty.
\end{align*}
The closedness of the one-dimensional weak derivative therefore
gives
$
u(\,\cdot\,\sigma)\in W^{1,p}(\rho,R)
$
and
$
\frac{d}{dr}u(r\sigma)=v(r\sigma)
$
for almost every $r\in(\rho,R)$ and for $\mu_K$-almost every
$\sigma\in\partial K$. Since $\rho$ and $R$ are arbitrary, the
conclusion holds locally on $(0,\infty)$. We define
$
\mathcal R_K(u)=v.
$
The convergence in \eqref{completion-R-convergence} shows that this
extension is continuous.

Finally, the map assigning to an element of
$\mathcal C^p_{K,a,b}$ its representative is injective.
Indeed, if the representative $u$ vanishes almost everywhere, then,
by the last part of the lemma,
$$
v(r\sigma)
=
\frac{d}{dr}u(r\sigma)
=
0
$$
for almost every $r>0$ and for $\mu_K$-almost every
$\sigma\in\partial K$. Hence, by \eqref{completion-u-convergence}
and \eqref{completion-R-convergence},
$$
\|u_n\|_{\mathcal C^p_{K,a,b}}\to0\ \text{as}\ n\to\infty.
$$
Therefore, the corresponding element of
$\mathcal C^p_{K,a,b}$ is the zero element.
\end{proof}

\begin{lemma}\label{lem:completion-identities}
Let $1<p<\infty$, $K\in\Ksc$, and 
$u\in\mathcal C_{K,a,b}^p\setminus\{0\}$. Then
$$
\int_{\RN}
\frac{|\mathcal R_K(u)|^p}{\|x\|_K^{bp}}dx
>0
\quad\text{and}\quad
\int_{\RN}
\frac{|u|^p}{\|x\|_K^{ap}}dx
>0.
$$

\begin{itemize}

\item[(i)]
Suppose that
$
b+1-a>0,
b\le\frac{N-p}{p}.
$
Then
$$
\int_{\RN}
\frac{|u|^p}
{\|x\|_K^{(p-1)a+b+1}}dx
<\infty,
$$
and
\begin{align}\label{completion-CKN-positive}
&\bigg(
\int_{\mathbb R^N}
\frac{|\mathcal R_K(u)|^p}{\|x\|_K^{bp}}dx
\bigg)^{\frac{1}{p}}
\bigg(
\int_{\mathbb R^N}
\frac{|u|^p}{\|x\|_K^{ap}}dx
\bigg)^{\frac{p-1}{p}}
-
\frac{N-1-(p-1)a-b}{p}
\int_{\mathbb R^N}
\frac{|u|^p}{\|x\|_K^{(p-1)a+b+1}}dx
\nonumber
\\
&=
\frac{1}{p}
\int_{\mathbb R^N}
\frac{1}{\|x\|_K^{bp}}
R_p
\Bigg(
-
\Bigg(
\frac{
\int_{\mathbb R^N}
\frac{|\mathcal R_K(u)|^p}{\|x\|_K^{bp}}dx
}{
\int_{\mathbb R^N}
\frac{|u|^p}{\|x\|_K^{ap}}dx
}
\Bigg)^{\frac{1}{p^2}}
u\|x\|_K^{b-a},
\Bigg(
\frac{
\int_{\mathbb R^N}
\frac{|u|^p}{\|x\|_K^{ap}}dx
}{
\int_{\mathbb R^N}
\frac{|\mathcal R_K(u)|^p}{\|x\|_K^{bp}}dx
}
\Bigg)^{\frac{p-1}{p^2}}
\mathcal R_K(u)
\Bigg)dx.
\end{align}

\item[(ii)]
Suppose that
$
b+1-a<0,
b\ge\frac{N-p}{p}.
$
Then
$$
\int_{\RN}
\frac{|u|^p}
{\|x\|_K^{(p-1)a+b+1}}dx
<\infty,
$$
and
\begin{align}\label{completion-CKN-negative}
&\bigg(
\int_{\mathbb R^N}
\frac{|\mathcal R_K(u)|^p}{\|x\|_K^{bp}}dx
\bigg)^{\frac{1}{p}}
\bigg(
\int_{\mathbb R^N}
\frac{|u|^p}{\|x\|_K^{ap}}dx
\bigg)^{\frac{p-1}{p}}
-
\frac{(p-1)a+b+1-N}{p}
\int_{\mathbb R^N}
\frac{|u|^p}{\|x\|_K^{(p-1)a+b+1}}dx
\nonumber
\\
&=
\frac{1}{p}
\int_{\mathbb R^N}
\frac{1}{\|x\|_K^{bp}}
R_p
\Bigg(
\Bigg(
\frac{
\int_{\mathbb R^N}
\frac{|\mathcal R_K(u)|^p}{\|x\|_K^{bp}}dx
}{
\int_{\mathbb R^N}
\frac{|u|^p}{\|x\|_K^{ap}}dx
}
\Bigg)^{\frac{1}{p^2}}
u\|x\|_K^{b-a},
\Bigg(
\frac{
\int_{\mathbb R^N}
\frac{|u|^p}{\|x\|_K^{ap}}dx
}{
\int_{\mathbb R^N}
\frac{|\mathcal R_K(u)|^p}{\|x\|_K^{bp}}dx
}
\Bigg)^{\frac{p-1}{p^2}}
\mathcal R_K(u)
\Bigg)dx.
\end{align}

\item[(iii)]
Suppose that $a=b+1$.
Then
\begin{align}\label{completion-critical-Hardy}
&\int_{\mathbb R^N}
\frac{|\mathcal R_K(u)|^p}{\|x\|_K^{bp}}dx
-
\bigg|
\frac{N-p(b+1)}{p}
\bigg|^p
\int_{\mathbb R^N}
\frac{|u|^p}{\|x\|_K^{p(b+1)}}dx\notag\\
&=
\int_{\RN}
\frac1{\|x\|_K^{bp}}
R_p\bigg(
-\frac{N-p(b+1)}p\frac{u}{\|x\|_K},
\mathcal R_K(u)
\bigg)dx.
\end{align}
\end{itemize}
\end{lemma}

\begin{proof}
Since $u\ne0$, one has
$$
\int_{\RN}
\frac{|u|^p}{\|x\|_K^{ap}}dx
>0.
$$
Suppose that
$$
\int_{\RN}
\frac{|\mathcal R_K(u)|^p}{\|x\|_K^{bp}}dx
=
0.
$$
Then
$
\mathcal R_K(u)=0
$
almost everywhere in $\RN$. By Lemma
\ref{lem:completion-representatives},
$$
\frac{d}{dr}u(r\sigma)=0
$$
for almost every $r>0$ and for $\mu_K$-almost every
$\sigma\in\partial K$. Hence,
$
u(r\sigma)=\Phi(\sigma)
$
for some measurable function $\Phi$ on $\partial K$.

If $\Phi\ne0$ on a set of positive $\mu_K$-measure, then
\begin{align*}
\int_{\RN}
\frac{|u|^p}{\|x\|_K^{ap}}dx
&=
\int_{\partial K}
|\Phi(\sigma)|^pd\mu_K(\sigma)
\int_0^\infty
r^{N-1-ap}dr
=
\infty,
\end{align*}
which contradicts
$u\in\mathcal C_{K,a,b}^p$. Therefore,
$\Phi=0$ almost everywhere and hence $u=0$, a contradiction. Thus,
$$
\int_{\RN}
\frac{|\mathcal R_K(u)|^p}{\|x\|_K^{bp}}dx
>0.
$$

We first prove part $(i)$. Let
$\{u_n\}\subset C_c^\infty(\RN\setminus\{o\})$ converge to $u$ in
$\mathcal C_{K,a,b}^p$. Applying the anisotropic $L^p$-CKN
inequality to $u_n-u_m$, we obtain
\begin{align*}
&\frac{N\!-\!1\!-\!(p\!-\!1)a\!-\!b}{p}\!
\int_{\RN}\!
\frac{|u_n-u_m|^p}
{\|x\|_K^{(p-1)a+b+1}}dx
\le
\bigg(
\int_{\RN}
\frac{|\mathcal R_K(u_n-u_m)|^p}
{\|x\|_K^{bp}}dx
\bigg)^{\frac1p}
\bigg(
\int_{\RN}
\frac{|u_n-u_m|^p}
{\|x\|_K^{ap}}dx
\bigg)^{\frac{p-1}{p}}.
\end{align*}
Consequently, $\{u_n\}$ is Cauchy in
$$
L^p\bigg(
\RN,
\frac{dx}{\|x\|_K^{(p-1)a+b+1}}
\bigg).
$$
Its limit agrees with the representative $u$ from Lemma
\ref{lem:completion-representatives}. Hence,
$$
\int_{\RN}
\frac{|u|^p}
{\|x\|_K^{(p-1)a+b+1}}dx
<\infty.
$$

Moreover,
\begin{align}
\int_{\RN}
\frac{|\mathcal R_K(u_n)|^p}{\|x\|_K^{bp}}dx
&\rightarrow
\int_{\RN}
\frac{|\mathcal R_K(u)|^p}{\|x\|_K^{bp}}dx\ \text{as}\ n\to\infty,
\label{completion-derivative-integral}\\
\int_{\RN}
\frac{|u_n|^p}{\|x\|_K^{ap}}dx
&\rightarrow
\int_{\RN}
\frac{|u|^p}{\|x\|_K^{ap}}dx\ \text{as}\ n\to\infty.
\label{completion-function-integral}
\end{align}
Therefore, for all sufficiently large $n$,
$$
\int_{\RN}
\frac{|\mathcal R_K(u_n)|^p}{\|x\|_K^{bp}}dx
>0
\quad\text{and}\quad
\int_{\RN}
\frac{|u_n|^p}{\|x\|_K^{ap}}dx
>0.
$$

By \eqref{completion-derivative-integral} and
\eqref{completion-function-integral}, one has
\begin{align}
&-\Bigg(
\frac{
\int_{\RN}
\frac{|\mathcal R_K(u_n)|^p}{\|x\|_K^{bp}}dx
}{
\int_{\RN}
\frac{|u_n|^p}{\|x\|_K^{ap}}dx
}
\Bigg)^{\frac1{p^2}}
u_n\|x\|_K^{b-a}
\rightarrow
-\Bigg(
\frac{
\int_{\RN}
\frac{|\mathcal R_K(u)|^p}{\|x\|_K^{bp}}dx
}{
\int_{\RN}
\frac{|u|^p}{\|x\|_K^{ap}}dx
}
\Bigg)^{\frac1{p^2}}
u\|x\|_K^{b-a}
\label{completion-first-argument}
\end{align}
and
\begin{align}
&\Bigg(
\frac{
\int_{\RN}
\frac{|u_n|^p}{\|x\|_K^{ap}}dx
}{
\int_{\RN}
\frac{|\mathcal R_K(u_n)|^p}{\|x\|_K^{bp}}dx
}
\Bigg)^{\frac{p-1}{p^2}}
\mathcal R_K(u_n)
\rightarrow
\Bigg(
\frac{
\int_{\RN}
\frac{|u|^p}{\|x\|_K^{ap}}dx
}{
\int_{\RN}
\frac{|\mathcal R_K(u)|^p}{\|x\|_K^{bp}}dx
}
\Bigg)^{\frac{p-1}{p^2}}
\mathcal R_K(u)
\label{completion-second-argument}
\end{align}
in
$
L^p\Big(
\RN,
\frac{dx}{\|x\|_K^{bp}}
\Big)\ \text{as}\ n\to\infty.
$
Indeed,
\begin{align*}
&\int_{\RN}
\frac{
|u_n\|x\|_K^{b-a}
-u\|x\|_K^{b-a}|^p
}
{\|x\|_K^{bp}}dx
=
\int_{\RN}
\frac{|u_n-u|^p}{\|x\|_K^{ap}}dx
\rightarrow0\ \text{as}\ n\to\infty.
\end{align*}

By Young's inequality,
$$
p|s|^{p-1}|t|
\le
(p-1)|s|^p+|t|^p.
$$
It follows from the definition of $R_p$ that
$$
0
\le
R_p(s,t)
\le
2|t|^p+2(p-1)|s|^p
\le
2p\big(|s|^p+|t|^p\big).
$$
By \eqref{completion-first-argument} and
\eqref{completion-second-argument}, the two arguments of $R_p$
converge in
$
L^p\Big(
\RN,
\frac{dx}{\|x\|_K^{bp}}
\Big).
$
After taking a subsequence if necessary, the two arguments converge
almost everywhere in $\RN$. By the continuity of $R_p$, the
corresponding remainders also converge almost everywhere. The estimate
above and the weighted $L^p$-convergence allow us to apply the
generalized dominated convergence theorem. Hence, taking the limit in
\eqref{f0-1} yields \eqref{completion-CKN-positive}.

The proof of part $(ii)$ is identical, using \eqref{f1-1} instead
of \eqref{f0-1}.

For part $(iii)$, let
$\{u_n\}\subset C_c^\infty(\RN\setminus\{o\})$ converge to $u$ in
$\mathcal C_{K,b+1,b}^p$. Then
\begin{align*}
-\frac{N-p(b+1)}p\frac{u_n}{\|x\|_K}
&\rightarrow
-\frac{N-p(b+1)}p\frac{u}{\|x\|_K},\\
\mathcal R_K(u_n)
&\rightarrow
\mathcal R_K(u)
\end{align*}
in
$
L^p\Big(
\RN,
\frac{dx}{\|x\|_K^{bp}}
\Big)\ \text{as}\ n\to\infty,
$
because
\begin{align*}
\int_{\RN}
\frac{
\big|
\frac{u_n-u}{\|x\|_K}
\big|^p
}
{\|x\|_K^{bp}}dx
&=
\int_{\RN}
\frac{|u_n-u|^p}
{\|x\|_K^{p(b+1)}}dx
\rightarrow0\ \text{as}\ n\to\infty.
\end{align*}
Using again the continuity and growth estimate for $R_p$, we may
pass to the limit in the weighted anisotropic Hardy identity. This
gives \eqref{completion-critical-Hardy}.
\end{proof}

The sharp constants in the inequalities \eqref{f0-2} and \eqref{f1-2} are generally not attained
in $C_c^\infty(\mathbb R^N\setminus\{o\})$. When $a\ne b+1$, they are attained in the
completion space $\mathcal C_{K,a,b}^p$ by the functions solving the
first-order equality equation associated with the $R_p$ remainder. When $a=b+1$, the sharp constant is not attained by any
nonzero function in the completion space.

\begin{theorem}\label{t4}
Let $1<p<\infty$, $K\in\Ksc$, and $a,b\in\R$.
Then the following statements hold.
\noindent
($i$) Assume that
$
b+1-a>0
\ \text{and}\
b\le \frac{N-p}{p}.
$
Then the constant
$$
C=\frac{N-1-(p-1)a-b}{p}
$$
in \eqref{f0-2} is sharp. For nonzero $u\in\mathcal C_{K,a,b}^p$, equality holds if and only if
$$
u(x)
=
\Phi\bigg(\frac{x}{\|x\|_K}\bigg)
\exp\bigg(
-\frac{\lambda}{b+1-a}\|x\|_K^{b+1-a}
\bigg),
\quad
\lambda>0,
$$
where
$
\Phi\in L^p(\partial K,d\mu_K)\setminus\{0\}.
$

\noindent
($ii$) Assume that
$
b+1-a<0
\ \text{and}\
b\ge \frac{N-p}{p}.
$
Then the constant
$$
C=\frac{(p-1)a+b+1-N}{p}
$$
in \eqref{f1-2} is sharp. For nonzero $u\in\mathcal C_{K,a,b}^p$, equality holds if and only if
$$
u(x)
=
\Phi\bigg(\frac{x}{\|x\|_K}\bigg)
\exp\bigg(
\frac{\lambda}{b+1-a}\|x\|_K^{b+1-a}
\bigg),
\quad
\lambda>0,
$$
where
$
\Phi\in L^p(\partial K,d\mu_K)\setminus\{0\}.
$

\noindent
($iii$)
Assume that $a=b+1$. Then the sharp constant in the product
CKN inequality is
$$
C
=
\frac{|N-p(b+1)|}{p},
$$
while the sharp constant in the corresponding weighted Hardy
inequality is
$$
C^p
=
\left|
\frac{N-p(b+1)}{p}
\right|^p.
$$
The sharp constant is not attained by any
$u\in\mathcal C^p_{K,b+1,b}\setminus\{0\}$.
\end{theorem}
\begin{proof}
\noindent
($i$)
Assume that
$
b+1-a>0
\ \text{and}\
b\le \frac{N-p}{p}.
$
Assume that
$u\in\mathcal C_{K,a,b}^p\setminus\{0\}$ attains equality in
\eqref{f0-2}. By \eqref{completion-CKN-positive}, one has
\begin{align*}
R_p\Bigg(
-\Bigg(
\frac{
\int_{\RN}
\frac{|\mathcal R_K(u)|^p}{\|x\|_K^{bp}}dx
}{
\int_{\RN}
\frac{|u|^p}{\|x\|_K^{ap}}dx
}
\Bigg)^{\frac1{p^2}}
u\|x\|_K^{b-a},
\Bigg(
\frac{
\int_{\RN}
\frac{|u|^p}{\|x\|_K^{ap}}dx
}{
\int_{\RN}
\frac{|\mathcal R_K(u)|^p}{\|x\|_K^{bp}}dx
}
\Bigg)^{\frac{p-1}{p^2}}
\mathcal R_K(u)
\Bigg)
=
0
\end{align*}
almost everywhere in $\RN$. It follows that
$
\mathcal R_K(u)
=
-\lambda\|x\|_K^{b-a}u,
$
where
$$
\lambda
=
\Bigg(
\frac{
\int_{\RN}
\frac{|\mathcal R_K(u)|^p}{\|x\|_K^{bp}}dx
}{
\int_{\RN}
\frac{|u|^p}{\|x\|_K^{ap}}dx
}
\Bigg)^{\frac1p}
>0.
$$

By Lemma \ref{lem:completion-representatives}, for
$\mu_K$-almost every $\sigma\in\partial K$, the function
$r\mapsto u(r\sigma)$ is locally absolutely continuous and satisfies
$$
\frac{d}{dr}u(r\sigma)
=
-\lambda r^{b-a}u(r\sigma)
$$
for almost every $r>0$. Solving this equation gives
$$
u(r\sigma)
=
\Phi(\sigma)
\exp\bigg(
-\frac{\lambda}{b+1-a}r^{b+1-a}
\bigg)
$$
for $\mu_K$-almost every $\sigma\in\partial K$. Therefore,
$$
u(x)
=
\Phi\bigg(
\frac{x}{\|x\|_K}
\bigg)
\exp\bigg(
-\frac{\lambda}{b+1-a}
\|x\|_K^{b+1-a}
\bigg).
$$

Since
$
a<b+1\le\frac Np,
$
one has $N-ap>0$. Hence,
\begin{align*}
\int_0^\infty\!
\exp\bigg(
\!-\!\frac{p\lambda}{b+1-a}r^{b+1-a}
\bigg)
r^{N-1-ap}dr
=
\frac1{b+1-a}
\bigg(
\frac{p\lambda}{b+1-a}
\bigg)^{-\frac{N-ap}{b+1-a}}
\Gamma\bigg(
\frac{N-ap}{b+1-a}
\bigg)
\in(0,\infty),
\end{align*}
here $\Gamma$ denotes the Gamma function.
By the anisotropic polar formula, one has
\begin{align*}
\int_{\partial K}
|\Phi(\sigma)|^pd\mu_K(\sigma)
\times
\int_0^\infty
\exp\bigg(
-\frac{p\lambda}{b+1-a}r^{b+1-a}
\bigg)
r^{N-1-ap}dr=\int_{\RN}
\frac{|u|^p}{\|x\|_K^{ap}}dx<\infty.
\end{align*}
Therefore,
$
\Phi\in L^p(\partial K,d\mu_K).
$
Moreover, $\Phi\ne0$ because $u\ne0$.

It remains to prove that these extremal functions belong to the completion
space. 
Since $
a<b+1\le \frac{N}{p},
$
then $N-ap>0$ and hence,
\begin{align*}
\int_{\mathbb R^N}|u|^p r^{-ap}dx
&=
\|\Phi\|_{L^p(\partial K,d\mu_K)}^p
\int_0^\infty
\exp\bigg(
-\frac{p\lambda}{b+1-a}r^{b+1-a}
\bigg)
r^{N-1-ap}dr
\\
&=
\|\Phi\|_{L^p(\partial K,d\mu_K)}^p
\frac{1}{b+1-a}
\bigg(
\frac{p\lambda}{b+1-a}
\bigg)^{-\frac{N-ap}{b+1-a}}
\Gamma\bigg(
\frac{N-ap}{b+1-a}
\bigg)
<\infty.
\end{align*}
Therefore,
$$
\int_{\mathbb R^N}
|\mathcal R_K(u)|^p r^{-bp}dx=\lambda^p\int_{\mathbb R^N}
 |u|^p r^{-ap}
dx<\infty.
$$

Let
$
\Phi_j=G_j|_{\partial K}$
such that
$
\Phi_j\to\Phi$
in $L^p(\partial K,d\mu_K)
$ for $G_j\in C^\infty(\mathbb R^N)$.
Since \(K\in\mathcal K^{ss}_{(o)}\), then
$
\sigma_K(x)\in C^1(\mathbb R^N\setminus\{o\})
$. Hence,
$\Phi_j(\sigma_K(x))
=
G_j(\sigma_K(x))
\in C^1(\mathbb R^N\setminus\{o\}).
$
Define
$$
u_j(x)
=
\Phi_j\bigg(\frac{x}{\|x\|_K}\bigg)
\exp\bigg(
-\frac{\lambda}{b+1-a}\|x\|_K^{b+1-a}
\bigg).
$$
Since $\Phi_j\circ\sigma_K$ is positively homogeneous of degree
zero, Euler's identity gives
$
\mathcal R_K(\Phi_j\circ\sigma_K)=0.
$
Consequently,
$
\mathcal R_K(u_j)-\mathcal R_K(u)
=
-\lambda\|x\|_K^{b-a}(u_j-u).
$
It follows that
$$
\int_{\mathbb R^N}|u_j-u|^p r^{-ap}dx
+
\int_{\mathbb R^N}
|\mathcal R_K(u_j)-\mathcal R_K(u)|^p r^{-bp}dx
\to0
$$
as $j\to\infty$. Thus, $
u_j\to u$
 as $j\to\infty$ in $\mathcal C^p_{K,a,b}$.

Let $0<\varepsilon<1$. 
Define
$\eta_\varepsilon\in C_c^\infty(0,\infty)$ such that
$0\le\eta_\varepsilon\le1$, 
$$\eta_\varepsilon(r)=
\begin{cases}
1, & \varepsilon\le r\le \varepsilon^{-1},\\
0, & 0<r\le \varepsilon^2 \ \text{or}\ r\ge \varepsilon^{-2},
\end{cases}$$
and
$$
|\eta_\varepsilon'(r)|
\le
\frac{M}{r\log(1/\varepsilon)}
\quad\text{on }
(\varepsilon^2,\varepsilon)\cup(\varepsilon^{-1},\varepsilon^{-2})
$$
where the constant $M>0$ is independent of $\varepsilon$.
Let
$
u_j^\varepsilon(x)
=
\eta_\varepsilon(\|x\|_K)u_j(x).
$
Let
$
\Omega_\varepsilon
=
\big\{
x\in\mathbb R^N:\frac{\varepsilon^2}{2}<\|x\|_K<\frac{2}{\varepsilon^2}
\big\}.
$
Then
$
\operatorname{supp} u_j^\varepsilon
\subset
\overline{\Omega_\varepsilon}
\subset
\mathbb R^N\setminus\{o\}.
$
Moreover, since
$
\Phi_j\big(\sigma_K(x)\big)\in C^1(\mathbb R^N\setminus\{o\})$ and $
\eta_\varepsilon(\|x\|_K)\in C^1(\mathbb R^N\setminus\{o\}),
$
we have
$
u_j^\varepsilon\in W^{1,p}_0(\Omega_\varepsilon).
$
Since $0\le \eta_\varepsilon\le1$ and
$\eta_\varepsilon(r)\to1$ as $\varepsilon\to 0$ for $r>0$, then the dominated convergence
theorem gives
$$
\int_{\mathbb R^N}
|u_j^\varepsilon-u_j|^p r^{-ap}dx
\to0\ \
\mathrm{and}\ \
\int_{\mathbb R^N}|\eta_\varepsilon(r)-1|^p
|\mathcal R_K(u_j)|^p
r^{-bp}dx
\to0
$$
as $\varepsilon\to0$.
As
$
\mathcal R_K(u_j^\varepsilon)-\mathcal R_K(u_j)
=
\big(\eta_\varepsilon(r)-1\big)\mathcal R_K(u_j)
+
\eta_\varepsilon'(r)u_j
$ and 
$
|c+d|^p\le 2^{p-1}\big(|c|^p+|d|^p\big),$ $c,d\in\R,$
then
\begin{align*}
\int_{\mathbb R^N}\!
|\mathcal R_K(u_j^\varepsilon)\!-\!\mathcal R_K(u_j)|^p r^{-bp}dx
\!\le\!
2^{p-1}\!\!\int_{\mathbb R^N}\!
|\eta_\varepsilon(r)\!-\!1|^p
|\mathcal R_K(u_j)|^p r^{-bp}dx
\!+\!
2^{p-1}\!\!\int_{\mathbb R^N}\!
|\eta_\varepsilon'(r)|^p
|u_j|^p r^{-bp}dx.
\end{align*}
As $\varepsilon\to 0$, one has
\begin{align*}
\int_{\{\varepsilon^2<\|x\|_K<\varepsilon\}}
|\eta_\varepsilon'(r)|^p|u_j|^p r^{-bp}dx
&\le
\frac{M^p\|\Phi_j\|_{L^p(\partial K,d\mu_K)}^p}{|\log\varepsilon|^p}
\int_{\varepsilon^2}^{\varepsilon}
\exp\bigg(
-\frac{p\lambda r^{b+1-a}}{b+1-a}
\bigg)
r^{N-1-p(b+1)}dr\\
&\le
\frac{M'}{|\log\varepsilon|^p}
\int_{\varepsilon^2}^{\varepsilon}\frac{dr}{r}
=
\frac{M'}{|\log\varepsilon|^{p-1}}
\to0\ \mathrm{for\ some\ constant}\ M'>0; 
\end{align*}
\begin{align*}
\int_{\{\varepsilon^{-1}\!<\|x\|_K<\varepsilon^{-2}\}}
\!|\eta_\varepsilon'(r)|^p|u_j|^p r^{-bp}dx
\!\le\!
\frac{M^p\|\Phi_j\|_{L^p(\partial K,d\mu_K)}^p}{|\log\varepsilon|^p}
\!\int_{\varepsilon^{-1}}^{\varepsilon^{-2}}\!
\!\exp\bigg(
\!-\!\frac{p\lambda r^{b+1-a}}{b+1-a}
\bigg)
r^{N-1-p(b+1)}dr
\to0.
\end{align*}
Hence,
$$
\int_{\mathbb R^N}
|u_j^\varepsilon-u_j|^p r^{-ap}dx
+
\int_{\mathbb R^N}
|\mathcal R_K(u_j^\varepsilon)-\mathcal R_K(u_j)|^p
r^{-bp}dx
\to0\ \mathrm{as}\ \varepsilon\to 0.
$$ 

Finally, since $u_j^\varepsilon$ is supported in $\Omega_\varepsilon$ and
the weights $r^{-ap}$ and $r^{-bp}$ are bounded above and
below on $\Omega_\varepsilon$, there exists a
sequence
$
w_{j,\varepsilon,n}\in C_c^\infty(\Omega_\varepsilon)
\subset C_c^\infty(\mathbb R^N\setminus\{o\})
$
such that
$$
\int_{\mathbb R^N}
|w_{j,\varepsilon,n}-u_j^\varepsilon|^p r^{-ap}dx
+
\int_{\mathbb R^N}
|\mathcal R_K(w_{j,\varepsilon,n}-u_j^\varepsilon)|^p r^{-bp}dx
\to0
$$
as $n\to\infty$.
Combining this with the convergences above and using a diagonal argument, we obtain
$
u\in \mathcal C_{K,a,b}^p.
$
Thus, the equality is attained
in $\mathcal C_{K,a,b}^p$, and the constant
$$
C=
\frac{|N-1-(p-1)a-b|}{p}
$$
is sharp.

Moreover,
$
\mathcal R_K(u)
=
-\lambda\|x\|_K^{b-a}u,
$
and hence
\begin{align*}
\int_{\RN}
\frac{|\mathcal R_K(u)|^p}{\|x\|_K^{bp}}dx
=
\lambda^p
\int_{\RN}
\frac{|u|^p}{\|x\|_K^{ap}}dx.
\end{align*}
Therefore, the parameter appearing in the $R_p$-remainder in
\eqref{completion-CKN-positive} is precisely $\lambda$, and the
remainder vanishes. Thus, every function in the stated family
attains equality, and the constant is sharp.

\noindent
($ii$)
Conversely, let
$\Phi\in L^p(\partial K,d\mu_K)\backslash\{0\}$.
Let
$\Phi_j=G_j|_{\partial K}$,
where
$G_j\in C^\infty(\mathbb R^N)$ and
$
\Phi_j\to\Phi
\ \text{in }L^p(\partial K,d\mu_K)
$ as $n\to\infty$
and define
$$
u_j(x)
=
\Phi_j\left(
\frac{x}{\|x\|_K}
\right)
\exp\left(
\frac{\lambda}{b+1-a}
\|x\|_K^{b+1-a}
\right).
$$
As in part ($i$),
$
\mathcal R_K(\Phi_j\circ\sigma_K)=0,
$
and therefore,
$
\mathcal R_K(u_j)-\mathcal R_K(u)
=
\lambda\|x\|_K^{b-a}(u_j-u).
$
Hence
$
u_j\to u
\ \text{in }\mathcal C^p_{K,a,b}
$ as $n\to\infty$.

It remains to verify the radial cut-off near the origin and at
infinity. Let $\eta_\varepsilon$ be the cut-off function used in
part ($i$).
Since
$
a>b+1\ge\frac Np,
$
one has $N-ap<0$.  Hence,
$$
\int_0^\infty
\exp\bigg(
\frac{p\lambda}{b+1-a}r^{b+1-a}
\bigg)
r^{N-1-ap}dr
\in(0,\infty).
$$
The anisotropic polar formula then gives
\begin{align*}
\int_{\partial K}
|\Phi(\sigma)|^pd\mu_K(\sigma)
\times
\int_0^\infty
\exp\bigg(
\frac{p\lambda}{b+1-a}r^{b+1-a}
\bigg)
r^{N-1-ap}dr=\int_{\RN}
\frac{|u|^p}{\|x\|_K^{ap}}dx<\infty.
\end{align*}
Thus,
$
\Phi\in L^p(\partial K,d\mu_K)\setminus\{0\}.
$

Conversely, let
$\Phi\in L^p(\partial K,d\mu_K)\setminus\{0\}$.
The angular approximation is the same as in part ($i$).
It remains to verify the radial cut-off near the origin and at
infinity.
Let $\eta_\varepsilon$ be the cut-off function used in part ($i$).
Since
$
b+1-a<0,
$
the exponential factor decays at the origin. More precisely, setting
$
c=-\frac{p\lambda}{b+1-a}>0,
d=a-b-1>0,
$
we have
\begin{align*}
&\int_{\{\varepsilon^2<\|x\|_K<\varepsilon\}}
|\eta_\varepsilon'(\|x\|_K)|^p
|u_j|^p
\|x\|_K^{-bp}dx
\le
\frac{M}{|\log\varepsilon|^p}
\int_{\varepsilon^2}^{\varepsilon}
e^{-cr^{-d}}
r^{N-1-p(b+1)}dr
\rightarrow0\ \text{as}\ \varepsilon\to0,
\end{align*} for some constant $M>0.$
On the other hand, since
$
b\ge\frac{N-p}{p},
$
one has
$
N-p(b+1)\le0.
$
Hence, for $r\ge1$,
$
r^{N-1-p(b+1)}\le r^{-1},
$
and therefore,
\begin{align*}
\int_{\{\varepsilon^{-1}<\|x\|_K<\varepsilon^{-2}\}}
|\eta_\varepsilon'(\|x\|_K)|^p
|u_j|^p
\|x\|_K^{-bp}dx
&\le
\frac{M}{|\log\varepsilon|^p}
\int_{\varepsilon^{-1}}^{\varepsilon^{-2}}
r^{N-1-p(b+1)}dr
\\
&\le
\frac{M}{|\log\varepsilon|^p}
\int_{\varepsilon^{-1}}^{\varepsilon^{-2}}\frac{dr}{r}\\
&=
\frac{M}{|\log\varepsilon|^{p-1}}
\rightarrow0\ \text{as}\ \varepsilon\to0,
\end{align*} for some constant $M>0.$
Thus the radial cut-off converges in
$\mathcal C^p_{K,a,b}$. The final smooth approximation on each
annulus is identical to that in part ($i$). Hence, the extremal belongs
to $\mathcal C^p_{K,a,b}$.

\noindent
($iii$) 
When $a=b+1$, the corresponding weighted anisotropic
$L^p$-Hardy inequality follows from Corollary
\ref{weighted-Hardy-identity} by taking
$
\lambda=bp.
$
To prove the sharpness of the constant, let
$$
u_\varepsilon(x)
=
\eta_\varepsilon(\|x\|_K)
\|x\|_K^{-\frac{N-p(b+1)}p}.
$$
By the anisotropic polar formula, one has
$$
\int_{\RN}
\frac{|u_\varepsilon|^p}{\|x\|_K^{p(b+1)}}dx
=
\mu_K(\partial K)
\int_0^\infty
|\eta_\varepsilon(r)|^p\frac{dr}{r},
$$
\begin{align*}
\int_{\RN}
\frac{|\mathcal R_K(u_\varepsilon)|^p}
{\|x\|_K^{bp}}dx
&=
\mu_K(\partial K)
\int_0^\infty
\bigg|
r\eta_\varepsilon'(r)
-
\frac{N-p(b+1)}{p}\eta_\varepsilon(r)
\bigg|^p
\frac{dr}{r}
\nonumber\\
&=
\mu_K(\partial K)
\bigg|\frac{N-p(b+1)}{p}\bigg|^p
\int_0^\infty
|\eta_\varepsilon(r)|^p\frac{dr}{r}
+
o\big(|\log\varepsilon|\big),
\end{align*}
and then 
$$
\frac{
\int_{\mathbb R^N}
\frac{|\mathcal R_K(u_\varepsilon)|^p}{\|x\|_K^{bp}}dx
}{
\int_{\mathbb R^N}
\frac{|u_\varepsilon|^p}{\|x\|_K^{p(b+1)}}dx
}
\to
\bigg|
\frac{N-p(b+1)}{p}
\bigg|^p\ \mathrm{as}\ \varepsilon\to0.
$$
By the same smoothing argument on annuli as
above, the same limit can be achieved by the functions in
$C_c^\infty(\mathbb R^N\setminus\{o\})$. Hence, the constant $\big|
\frac{N-p(b+1)}{p}
\big|^p$
is sharp.

It remains to prove that the sharp constant is not attained by a
nonzero function in $\mathcal C_{K,b+1,b}^p$. Suppose, to the
contrary, that
$
u\in\mathcal C_{K,b+1,b}^p\setminus\{0\}
$
attains equality.
By \eqref{completion-critical-Hardy}, the corresponding
$R_p$-remainder must vanish. Hence,
$$
\mathcal R_K(u)
=
-\frac{N-p(b+1)}p\frac{u}{\|x\|_K}
$$
almost everywhere in $\RN$.
By Lemma \ref{lem:completion-representatives}, for
$\mu_K$-almost every $\sigma\in\partial K$,
$$
\frac{d}{dr}u(r\sigma)
=
-\frac{N-p(b+1)}p\frac{u(r\sigma)}r
$$
for almost every $r>0$. Therefore,
$
u(r\sigma)
=
\Phi(\sigma)r^{-\frac{N-p(b+1)}p}
$
for some measurable function $\Phi$ on $\partial K$. Since
$u\ne0$, the function $\Phi$ is nonzero on a set of positive
$\mu_K$-measure. However,
\begin{align*}
\int_{\RN}
\frac{|u|^p}{\|x\|_K^{p(b+1)}}dx
=
\int_{\partial K}
|\Phi(\sigma)|^pd\mu_K(\sigma)
\int_0^\infty\frac{dr}{r}=
\infty.
\end{align*}
This contradicts
$u\in\mathcal C_{K,b+1,b}^p$. Hence, the sharp constant is not
attained by any nonzero function in the completion space.
If $N=p(b+1)$, then $C=0$. Nevertheless, by Lemma
\ref{lem:completion-identities}, both factors on the left-hand side
of the product inequality are strictly positive for every
$u\in\mathcal C^p_{K,b+1,b}\setminus\{0\}$.
Hence, equality still cannot occur for a nonzero function.
\end{proof}
\begin{remark}
The condition
$
b\le\frac{N-p}{p}
$
in Theorem \ref{t4} (i) is essential for the attainment statement.
Indeed, suppose that
$
b>\frac{N-p}{p},
$
and let $\eta\in C^\infty(0,\infty)$ vanish near $0$ and satisfy
$\eta(\delta)=1$. Then H\"older's inequality gives
\begin{align*}
1
&=
|\eta(\delta)-\eta(0)|
\le
\int_0^\delta|\eta'(r)|dr
\le
\left(
\int_0^\delta
|\eta'(r)|^p r^{N-1-bp}dr
\right)^{\frac1p}
\left(
\int_0^\delta
r^{-\frac{N-1-bp}{p-1}}dr
\right)^{\frac{1}{p'}}.
\end{align*}
The last integral is finite precisely when
$
b>\frac{N-p}{p}.
$
In this case,
$$
\int_0^\delta
r^{-\frac{N-1-bp}{p-1}}dr
=
C\delta^{\frac{p(b+1)-N}{p-1}},
$$
and therefore
$$
\int_0^\delta
|\eta'(r)|^p r^{N-1-bp}dr
\ge
c\delta^{N-p(b+1)}
\rightarrow\infty
\quad
\text{as }\delta\to0.
$$
Thus, if $b>\frac{N-p}{p}$, a function with a nonvanishing angular
factor at the origin cannot in general be approximated in
$\mathcal C^p_{K,a,b}$ by functions supported away from the origin.
Hence the condition $b\le\frac{N-p}{p}$ is optimal for the attainment
statement in Theorem \ref{t4} (i).
\end{remark}

We next record the anisotropic $L^p$-Heisenberg uncertainty
principle. The inequalities below are valid for every $1<p<\infty$.
When $p\le N$, the corresponding extremal functions belong to the
completion space
$\mathcal C^p_{K,-\frac{1}{p-1},0}$.
\begin{corollary}\label{cor:radial-HUP}
Let $1<p<\infty$, $p'=\frac{p}{p-1}$, and $K\in\Ksc$.
Then, for  $u\in C_c^\infty(\mathbb R^N\setminus\{o\})$, one has
\begin{align}\label{HUP-additive}
\int_{\mathbb R^N}
|\mathcal R_K(u)|^pdx
+
(p-1)
\int_{\mathbb R^N}
\|x\|_K^{p'}|u|^pdx
\ge
N
\int_{\mathbb R^N}|u|^pdx,
\end{align}
\begin{align}\label{HUP-product}
\bigg(
\int_{\mathbb R^N}
|\mathcal R_K(u)|^pdx
\bigg)^{\frac1p}
\bigg(
\int_{\mathbb R^N}
\|x\|_K^{p'}|u|^pdx
\bigg)^{\frac1{p'}}
\ge
\frac{N}{p}
\int_{\mathbb R^N}|u|^pdx.
\end{align}
The sharp constants are $N$ and $\frac{N}{p}$, respectively. If, in addition, $p\le N$, then the sharp constants are
attained in $\mathcal C_{K,-\frac1{p-1},0}^p$. Equality in
\eqref{HUP-additive} is attained by
$$
u(x)=\Phi\bigg(\frac{x}{\|x\|_K}\bigg)
\exp\bigg(
-\frac{1}{p'}\|x\|_K^{p'}
\bigg),
$$
while equality in \eqref{HUP-product} is attained by
$$
u(x)=\Phi\bigg(\frac{x}{\|x\|_K}\bigg)
\exp\bigg(
-\frac{1}{p'\lambda^{p'}}\|x\|_K^{p'}
\bigg),
\quad \lambda>0,
$$
where
$
\Phi\in
L^p(\partial K,d\mu_K)\setminus\{0\}.
$
\end{corollary}

\begin{proof}
From the anisotropic $L^p$-Heisenberg-type identity \eqref{f7} and the
nonnegativity of $R_p$, we immediately obtain \eqref{HUP-additive}. Equality
in \eqref{HUP-additive} holds if and only if the remainder in \eqref{f7}
vanishes, namely
$$
R_p\bigg(
-u\|x\|_K^{\frac{1}{p-1}},
\mathcal R_K(u)
\bigg)=0.
$$
Since $R_p(s,t)=0$ if and only if $s=t$, this is equivalent to
$$
\mathcal R_K(u)
=
-\|x\|_K^{\frac{1}{p-1}}u.
$$
Solving this equation, we obtain
$$
u(x)
=
\Phi\bigg(
\frac{x}{\|x\|_K}
\bigg)
\exp\bigg(
-\frac1{p'}\|x\|_K^{p'}
\bigg),
$$
where
$
\Phi\in L^p(\partial K,d\mu_K).
$
If $u\not\equiv0$, then $\Phi\not\equiv0$.

Next, \eqref{HUP-product} follows from the anisotropic $L^p$-CKN identity
\eqref{f8}. Equality in \eqref{HUP-product} holds if and only if the
corresponding $R_p$-remainder in \eqref{f8} vanishes. Equivalently, for some
$\lambda>0$,
$$
\mathcal R_K(u)
=
-\lambda^{-p'}\|x\|_K^{\frac{1}{p-1}}u.
$$
Solving this equation gives
$$
u(x)
=
\Phi\bigg(
\frac{x}{\|x\|_K}
\bigg)
\exp\bigg(
-\frac1{p'\lambda^{p'}}
\|x\|_K^{p'}
\bigg),
\quad
\lambda>0,
$$
where
$
\Phi\in L^p(\partial K,d\mu_K).
$
If $u\not\equiv0$, then $\Phi\not\equiv0$.

Therefore, the inequalities \eqref{HUP-additive} and \eqref{HUP-product}
follow from \eqref{f7} and \eqref{f8}, respectively, for every
$1<p<\infty$.

If $p\le N$, then
$
a=-\frac{1}{p-1},
b=0
$
satisfy
$
b+1-a=p'>0,
b\le\frac{N-p}{p}.
$
Hence, Theorem \ref{t4} ($i$) applies. The corresponding equality
equations give precisely the functions displayed above, and the
radial cut-off and density argument in the proof of
Theorem \ref{t4} ($i$) shows that they belong to
$\mathcal C^p_{K,-\frac{1}{p-1},0}$.
\end{proof}

We next derive norm-based anisotropic gradient inequalities by using
the origin-symmetric convex body
$
L
=
\operatorname{conv}\big(
K\cup(-K)
\big).
$

\begin{theorem}\label{t5}
Let $1<p<\infty$, $K\in\Ksc$,  
$
L
=
\operatorname{conv}\big(
K\cup(-K)
\big),
L^*
=
K^*\cap(-K^*),
$ $a,b\in\R$
and $
u\in C_c^\infty\big(
\RN\setminus\{o\}
\big).
$
Suppose first that either
$
b+1-a>0,
b\le\frac{N-p}{p},
$
or
$
b+1-a<0,
b\ge\frac{N-p}{p}.
$
Then
\begin{align}
&\int_{\RN}
\frac{\|\nabla u\|_{L^*}^p}
{\|x\|_K^{bp}}dx
+
(p-1)
\int_{\RN}
\frac{|u|^p}
{\|x\|_K^{ap}}dx
\ge
\big|
N-1-(p-1)a-b
\big|
\int_{\RN}
\frac{|u|^p}
{\|x\|_K^{(p-1)a+b+1}}dx.
\label{gradient-Hardy}
\end{align}
Moreover,
\begin{align}
&\bigg(
\int_{\RN}
\frac{\|\nabla u\|_{L^*}^p}
{\|x\|_K^{bp}}dx
\bigg)^{\frac1p}
\bigg(
\int_{\RN}
\frac{|u|^p}
{\|x\|_K^{ap}}dx
\bigg)^{\frac{p-1}{p}}
\ge
\frac{
\big|
N-1-(p-1)a-b
\big|
}{p}
\int_{\RN}
\frac{|u|^p}
{\|x\|_K^{(p-1)a+b+1}}dx.
\label{gradient-CKN}
\end{align}
If $a=b+1$, then
\begin{align}
&\int_{\RN}
\frac{\|\nabla u\|_{L^*}^p}
{\|x\|_K^{bp}}dx
+
(p-1)
\int_{\RN}
\frac{|u|^p}
{\|x\|_K^{p(b+1)}}dx
\ge
\bigg[
p-1
+
\bigg|
\frac{N-p(b+1)}{p}
\bigg|^p
\bigg]
\int_{\RN}
\frac{|u|^p}
{\|x\|_K^{p(b+1)}}dx,
\label{critical-gradient-Hardy}
\end{align}
and
\begin{align}
&\bigg(
\int_{\RN}
\frac{\|\nabla u\|_{L^*}^p}
{\|x\|_K^{bp}}dx
\bigg)^{\frac1p}
\bigg(
\int_{\RN}
\frac{|u|^p}
{\|x\|_K^{p(b+1)}}dx
\bigg)^{\frac{p-1}{p}}
\ge
\frac{
|N-p(b+1)|
}{p}
\int_{\RN}
\frac{|u|^p}
{\|x\|_K^{p(b+1)}}dx.
\label{critical-gradient-CKN}
\end{align}

If, in addition, $K$ is origin-symmetric, then $L=K$, and all the
constants in
\eqref{gradient-Hardy}-\eqref{critical-gradient-CKN}
are sharp.
\end{theorem}

\begin{proof}
By \eqref{two-sided-anisotropic-CS}, one has
\begin{align}
|\mathcal R_K(u)(x)|
&=
\frac{
|x\cdot\nabla u(x)|
}{
\|x\|_K
}
\le
\|\nabla u(x)\|_{L^*}
\label{gradient-radial-control}
\end{align}
for almost every $x\in\RN$.

Suppose first that one of the two noncritical parameter conditions
above holds. Combining
\eqref{gradient-radial-control} with
\eqref{f4-1} and \eqref{f4-2}, according to the corresponding
parameter region, gives \eqref{gradient-Hardy}. Similarly,
combining \eqref{gradient-radial-control} with
\eqref{f0-2} and \eqref{f1-2} gives
\eqref{gradient-CKN}.

Suppose next that $a=b+1$.
Taking $\lambda=bp$ in Corollary \ref{weighted-Hardy-identity}
and using \eqref{gradient-radial-control}, we obtain
$$
\int_{\mathbb R^N}
\frac{\|\nabla u\|_{L^*}^p}{\|x\|_K^{bp}}dx
\ge
\left|
\frac{N-p(b+1)}{p}
\right|^p
\int_{\mathbb R^N}
\frac{|u|^p}{\|x\|_K^{p(b+1)}}dx.
$$
Adding
$$
(p-1)
\int_{\mathbb R^N}
\frac{|u|^p}{\|x\|_K^{p(b+1)}}dx
$$
to both sides gives \eqref{critical-gradient-Hardy}.
Combining \eqref{gradient-radial-control} with the critical
anisotropic $L^p$-CKN inequality in Theorem \ref{t4} ($iii$) gives
\eqref{critical-gradient-CKN}.

It remains to prove sharpness when $K$ is origin-symmetric. In this
case,
$
L=K,
L^*=K^*.
$
We first consider the case $a\ne b+1$. Choose $\Phi$ in the
extremal functions from Theorem \ref{t4} to be a nonzero constant
and take $\lambda=1$. Then the resulting extremal function is radial with respect to $K$.
Thus, using the same symbol $u$, we write
$$
u(x)=u(r),
\quad
r=\|x\|_K.
$$
For such a radial function,
$
\nabla u(x)
=
u'(r)\nabla\|x\|_K,
$
and hence,
\begin{align}
\|\nabla u(x)\|_{L^*}
&=
\|\nabla u(x)\|_{K^*}\notag\\
&=
|u'(r)|\cdot
\big\|\nabla\|x\|_K\big\|_{K^*}\notag\\
&=
|u'(r)|\notag\\
&=
|\mathcal R_K(u)(x)|.
\label{radial-gradient-equality}
\end{align}
If
$
b+1-a>0,
b\le\frac{N-p}{p},
$
then
$
\mathcal R_K(u)
=
-\|x\|_K^{b-a}u,
$
while if
$
b+1-a<0,
b\ge\frac{N-p}{p},
$
then
$
\mathcal R_K(u)
=
\|x\|_K^{b-a}u.
$
Therefore, the corresponding $R_p$-remainders in both the additive
and multiplicative radial identities vanish. By
\eqref{radial-gradient-equality}, the gradient terms coincide with
the radial-derivative terms for these radial functions.

The same radial cut-off and density argument used in the proof of
Theorem \ref{t4} produces a sequence in
$
C_c^\infty\big(
\RN\setminus\{o\}
\big)
$
approaching each radial extremal in the corresponding weighted
full-gradient norm. For radial cut-offs, equality
\eqref{radial-gradient-equality} remains valid. For each fixed
cut-off, the final smooth approximation follows from the
$W^{1,p}$-density theorem on an annulus, where the weights are
bounded above and below and $\|\cdot\|_{K^*}$ is equivalent to the
Euclidean norm. Consequently, the constants in
\eqref{gradient-Hardy} and \eqref{gradient-CKN} are sharp.

Finally, suppose that $a=b+1$. Let $\{u_\varepsilon\}$ be the radial
extremizing sequence constructed in the proof of Theorem
\ref{t4}(iii). By \eqref{radial-gradient-equality},
\begin{align*}
\int_{\RN}
\frac{\|\nabla u_\varepsilon\|_{K^*}^p}
{\|x\|_K^{bp}}dx
=
\int_{\RN}
\frac{|\mathcal R_K(u_\varepsilon)|^p}
{\|x\|_K^{bp}}dx.
\end{align*}
Thus, the same limiting quotients as in the radial inequalities give
the constants in \eqref{critical-gradient-Hardy} and
\eqref{critical-gradient-CKN}. Choosing the smooth approximations
sufficiently close and taking a diagonal sequence yields extremizing
sequences in
$
C_c^\infty\big(
\RN\setminus\{o\}
\big).
$
Hence, the critical constants are also sharp.
\end{proof}
\begin{remark}
Even if $K$ is smooth and strictly convex, the origin-symmetric
convex body
$$
L^*
=
K^*\cap(-K^*)
$$
need not be smooth or strictly convex. This causes no difficulty,
since the proof of Theorem \ref{t5} uses only the norm
$\|\cdot\|_{L^*}$ and does not require any differentiability of its
Minkowski functional.
\end{remark}

\begin{remark}
When $K$ is not origin-symmetric, we do not claim that the constants
in Theorem \ref{t5} are sharp. The extremal functions for the radial
inequalities are radial with respect to $K$, whereas equality in
\eqref{gradient-radial-control} depends on the support geometry of
the symmetrized polar body $L^*$. These two equality conditions need
not be compatible in the non-symmetric setting.
\end{remark}

\bigskip
\noindent \textbf{Acknowledgments.}
The author would like to express her sincere gratitude to Professor Nguyen Lam and Professor Deping Ye for their valuable guidance, helpful discussions, and insightful suggestions during the preparation of this work.

\bigskip

\noindent
\textit{Department of Mathematics and Statistics,
Memorial University of Newfoundland,
St. John's, Newfoundland and Labrador A1C 5S7, Canada.}

\noindent
\textit{Email address:}
\texttt{zhenzhenw@mun.ca}

\end{document}